\documentclass[11pt]{article}
\usepackage{geometry}   
\usepackage{graphicx,amsthm,pstricks-add,matnot}
\usepackage{amssymb,amsmath,enumitem,xcolor,textcomp,multicol}

\usepackage[square,numbers]{natbib}

\newtheorem{definition}{Definition}
\newtheorem{theorem}{Theorem}
\newtheorem{lemma}{Lemma}

\author{Peter Staab, Charles Fisher, Mark Maggio, Michael Andrade, Erin Farrell and Haley Schilling}
\title{The Magic of Permutation Matrices:  Categorizing, Counting and Eigenspectra of Magic Squares}

\begin{document}
\maketitle

\section{Introduction}

A magic square is an arrangement of a set of integers into a square array such that each row, column and the two diagonals add to the same value called the magic number.  Magic squares generally fall into the realm of recreational mathematics \cite{Pasles:2008,pickover:2002}, however a few times in the past century and more recently, they have become the interest of more-serious mathematicians.  In this paper we explore how permutation matrices play a key role in categorizing and enumerating magic squares.  Also, we generalize the results of Mattingly \cite{mattingly:2000} to find the determinant and eigenspectra of sets of magic squares that have various symmetries.

Throughout the extensive history of magic squares (see \cite{Pasles:2008}), the main interest lay in their construction.  Only in the past 60 years has analysis of magic squares as matrices been performed.   In 1917, Henry Ernest Dudeney \cite{dudeney:1917} classified all of the 4 by 4 magic squares, mainly by symmetries.  Figure \ref{fig:order4-class} (shown on the following page) shows Dudeney's classification scheme.  There are 12 classifications of magic squares of this size and the arcs shown in the figure for each type square signify the two cells that add to half the magic number. For example, the \textbf{D\"urer Magic Square}, named after Albrecht D\"urer (1471--1526), is 
\begin{align}
\begin{tabular}{|c|c|c|c|} \hline
16 &	3  &	2  &	13 \\ \hline
5 &	10 & 	11  &	8 \\ \hline
9 & 	6  &	7  &	12 \\ \hline
4  &	15 & 	14  &	1 \\ \hline 
\end{tabular}
\label{durer:magic:square}
\end{align}
and it can be classified as a type III (or regular) magic square.  This can be seen by adding up elements that are symmetric about the center or are linked in the Dudeney diagram.   Each symmetric pair adds to 17 (half the magic number of 34).  

\begin{figure}
\begin{center}
\includegraphics[width=3in,angle=270]{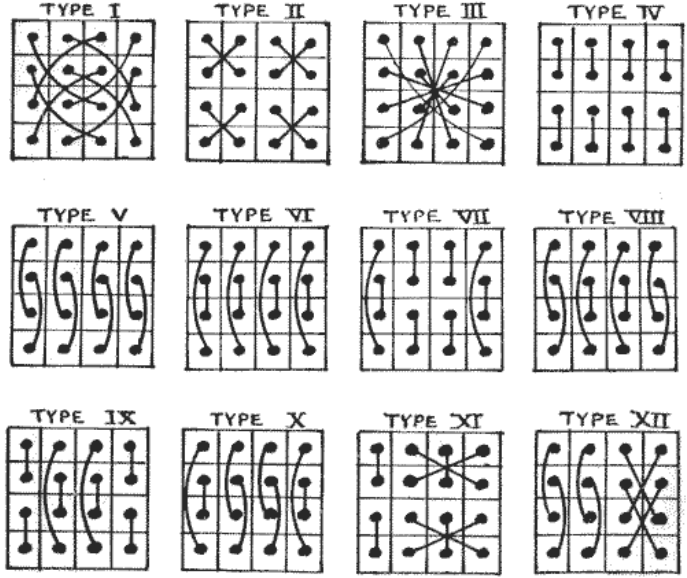}
\end{center}
\caption{The classification scheme of Henry Ernest Dudeney \cite{dudeney:1917} for 4 by 4 natural magic squares.  The arcs denote the pairs of cells that add up to half the magic number.  In this case, the magic number is 34. We will show that magic square of types I--VI, XI and XII satisfy equations containing permutation matrices and that generalized versions of types I--VI are singular. } 
\label{fig:order4-class}
\end{figure}

In 1948, C. W. Trigg provided \cite{trigg:1948} a way to calculate the determinant of each magic square in Dudeney's classification.  In particular, he showed that each one of the 4 by 4 magic squares of Dudeney's type I--VI has determinant zero. In 2000, R. Bruce Mattingly, proved \cite{mattingly:2000} that not only is the determinant of each type III or regular 4 by 4 magic square equal to zero, but this is true for any even-sized regular magic square.

Mattingly's analysis used the reverse matrix to define a regular magic square, and used the definition to show some interesting properties about regular even-ordered magic squares.  We expand the analysis to a larger set.  In fact, we generalize Dudeney's type I--VI to higher-order magic squares and prove that they are all singular and show many properties of their eigenspectra.  

\section{Mathematical Background} \label{sect:math:back}

It is mathematically convenient to define magic squares in a more precise manner than discussed above.  First, let the column vector $\be_i$ be the standard basis vector in $\mathbb{R}^n$, the vector with all zeros except the $i$th element, which is one.  Also, let $\be$ be the vector in $\mathbb{R}^n$ with every element 1,  and let $\bJ$ be the square matrix of size $n$ defined by
\begin{align}
\bJ_{i,j} =\begin{cases}
	 1, & \text{if $i+j=n+1$}, \\
	 0, & \text{otherwise}.
\end{cases}
\label{def:jmatrix}
\end{align}

The matrix $\bJ$, the identity matrix reflected either horizontally or vertically, is often called the \emph{reverse matrix} since it is useful for reversing rows or columns of vectors and matrices.  If $\bA$ is a matrix, then $\bJ \bA$ is the matrix $\bA$ reflected vertically and $\bA \bJ$ is $\bA$ reflected horizontally.  A specific example of this is shown below.  

Since we are interested in understanding how permutation matrices transform magic squares, we define a magic square using matrix operations. 

\begin{definition}
	A \textbf{magic square of order} $\mathbf{n}$ is an $n$ by $n$ matrix $\bA$ consisting of real numbers with the following properties:
\begin{align}
\be_i^T \bA \be & = \mu, ~~~~~ \text{for $i=1,2, \ldots, n$,} \label{def:magic:square:row} \\
\be^T \bA \be_j &  = \mu,~~~~~ \text{for $j=1,2, \ldots, n$, }
\label{def:magic:square:col} \\
\tr(\bA) & = \mu, \label{def:magic:square:diag} \\
\tr(\bJ\bA) & = \mu,  \label{def:magic:square:altdiag}  
\end{align}
where $\mu$ is called the \textbf{magic number}.

If the numbers used in the magic square are the consecutive integers, $1,2, \ldots, n^2$, then $\mu=n(n^2+1)/2$ and the magic square is called \textbf{natural}.  
\end{definition}

Equations (\ref{def:magic:square:row}) and (\ref{def:magic:square:col}) are the properties that the rows and columns respectively must sum to $\mu$.  Equations (\ref{def:magic:square:diag}) and (\ref{def:magic:square:altdiag}) are the properties that main and minor diagonals must also sum to $\mu$.  Writing the definition in this way allows proofs to proceed in a clearer manner.  


Although the numbers within magic squares have typically been defined on the integers (often it is important that they are non-repeating as will be indicated), the definition above and many of the results that follow are generally applied to magic squares on the reals.  

Since some of the results are applicable to semi-magic squares, we define that here.  

\begin{definition}
 A \textbf{semi-magic square} is a $n$ by $n$ matrix $\bA$ of reals that satisfies (\ref{def:magic:square:row}) and (\ref{def:magic:square:col}).   If the elements of the matrix consist of a reordering of the numbers $1,2,\ldots,n^2$, then the result is called a \textbf{natural semi-magic square}.  
\end{definition}

\subsection{Symmetries of Matrices}

As is standard, a matrix $\bA$ is symmetric if $\bA=\bA^T$.  Permutation matrices, and in particular permutation matrices with particular symmetries, play an important role in this study.  

\begin{definition}
A square matrix $\bA$ is \textbf{persymmetric} if it satifies $\bJ \bA = (\bJ \bA)^T$.  A square matrix is \textbf{bisymmetric} if it is both symmetric and persymmetric.  A square matrix is called \textbf{singly symmetric} if it is either symmetric or persymmetric but not both.    
\end{definition}

The term persymmetric means that it is symmetric with respect to the minor diagonal.  A result of these is that $\bJ\bA\bJ=\bA^T$, which can be interpreted as the property that rotating the matrix halfway around is equal to its transpose.  We will investigate bisymmetric permutation matrices in the next section. 

Another class of permutation matrices that play a role in magic squares are those with a 90\textdegree\, rotational symmetry.  For example, 
\begin{align}
\begin{bmatrix}
0 & 1 & 0 & 0 \\
0 & 0 & 0 & 1 \\
1 & 0 & 0 & 0 \\
0 & 0 & 1 & 0 
\end{bmatrix} \label{eq:rot90:perm:ex}
\end{align}
is such a matrix.  More formally, we define the following:
\begin{definition}
If a matrix $\bA$ satisfies $\bA^T = \bJ\bA$, then $\bA$ is said to be \textbf{90\textdegree-symmetric}.  
\end{definition}

We will see later that permutation matrices that are bisymmetric or 90\textdegree-symmetric transform magic squares into other magic squares.

\subsection{Magic Squares formed by rotating and reflecting}

It is a standard known ``fact'' that for each magic square, there are an additional 7 magic squares found by reflecting and rotating the original one.  Most who study magic squares consider the eight related squares as equivalent or isomorphic.  However, if thought of as a matrix, then these 8 magic squares are different but still all magic squares.  As has been seen, using matrix operations is important and the eight related  magic squares can be written as $\bA,$ $\bJ\bA$, $\bA\bJ$, $\bJ\bA\bJ$, $\bA^T$, $\bJ \bA^T$, $\bA^T\bJ$, $\bJ \bA^T \bJ$. 

For example, if $\bA$ is the D\"urer magic square in (\ref{durer:magic:square}) and shown below, then 
\begin{align*} 
\bA & = \begin{bmatrix}
16 & 3 & 2 & 13 \\
5 & 10 & 11 & 8 \\
9 & 6 & 7 & 12 \\
4 & 15 & 14 & 1
\end{bmatrix}, & 
\bJ\bA & = \begin{bmatrix}
4 & 15 & 14 & 1 \\
9 & 6 & 7 & 12 \\
5 & 10 & 11 & 8 \\
16 & 3 & 2 & 13
\end{bmatrix}, & 
\bA \bJ & = \begin{bmatrix}
13 & 2 & 3 & 16 \\
8 & 11 & 10 & 5 \\
12 & 7 & 6 & 9 \\
1 & 14 & 15 & 4
\end{bmatrix}, \\ 
\bJ\bA\bJ & = \begin{bmatrix}
1 & 14 & 15 & 4 \\
12 & 7 & 6 & 9 \\
8 & 11 & 10 & 5 \\
13 & 2 & 3 & 16
\end{bmatrix}, &
\bA^T & = \begin{bmatrix}
16 & 5 & 9 & 4 \\
3 & 10 & 6 & 15 \\
2 & 11 & 7 & 14 \\
13 & 8 & 12 & 1
\end{bmatrix}, &
\bJ \bA^T & = \begin{bmatrix}
13 & 8 & 12 & 1 \\
2 & 11 & 7 & 14 \\
3 & 10 & 6 & 15 \\
16 & 5 & 9 & 4
\end{bmatrix}, \\
 \bA^T \bJ & = \begin{bmatrix}
4 & 9 & 5 & 16 \\
15 & 6 & 10 & 3 \\
14 & 7 & 11 & 2 \\
1 & 12 & 8 & 13
\end{bmatrix}, &
\bJ \bA^T \bJ & = \begin{bmatrix}
1 & 12 & 8 & 13 \\
14 & 7 & 11 & 2 \\
15 & 6 & 10 & 3 \\
4 & 9 & 5 & 16
\end{bmatrix}, 
\end{align*}
are the unique rotations and reflections of the D\"urer magic square.  Starting with any magic square of any order, it may seem obvious that these 7 transforms result in a magic square.   We will show this below, not to prove an obvious result, but mainly to illustrate how the definition of the magic square in (\ref{def:magic:square:row})--(\ref{def:magic:square:altdiag}) is useful in such proofs.  

\begin{theorem} \label{thm:8magics}
	If $\bA$ is a magic square then following are also magic squares: $\bJ\bA$, $\bA\bJ$, $\bJ\bA\bJ$, $\bA^T$, $\bJ \bA^T$, $\bA^T\bJ$, $\bJ \bA^T \bJ$. 
\end{theorem}

\begin{proof}
	First, we show that (\ref{def:magic:square:row})--(\ref{def:magic:square:altdiag}) are satisfied for $\bJ\bA$.  Applying (\ref{def:magic:square:row}) results in
	\begin{align*}
	\be_i^T (\bJ\bA) \be & = (\be_i^T \bJ) \bA \be = \be_{n+1-i}^T \bA \be = \mu, 
\end{align*}
for $i=1,2,\ldots,n$. In this case $\be_i^T\bJ=\be_{n+1-i}^T$ because the reversed version of $\be_i$ is $\be_{n+1-i}$. 

   Similarly,
   \begin{align*}
	\be^T (\bJ\bA) \be_i & = (\be^T \bJ) \bA \be_i = \be^T \bA \be_i = \mu, 
\end{align*}
for all $i=1,2,\ldots, n$ and $\be^T \bJ=\be^T$.  Property (\ref{def:magic:square:diag}) applied to $\bJ \bA$ is (\ref{def:magic:square:altdiag}) and vice versa.  This shows that $\bJ \bA$ is a magic square and the proofs of the remaining matrices listed in the theorem are similar to this.  

\end{proof}

\subsection{Eigenvalues and Eigenvectors of Magic Squares}

There are a few basics of eigenvalues and eigenvectors of magic squares.  Much of this work appears in Mattingly \cite{mattingly:2000} and is reproduced to demonstrate the proofs using the matrix definitions of the magic squares. Also, Loly et al. \cite{loly:2009} gives an extensive review of the literature of both applicable matrix results as well as properties of magic squares.  They emphasize eigenvalues and eigenvectors of general magic squares with a particular consideration on the characteristic polynomial.  They also show many examples of specific low-order magic squares.  

Again, to illustrate the usefulness of the definition of the magic square given above, we show the following. 

\begin{lemma} \label{lem:magic:eigenvalue}
Let $\bA$ be a magic square with magic number $\mu$. 
\begin{itemize}
\item The vector $\be$ is an eigenvector of $\bA$ with corresponding eigenvalue $\mu$.  
\item The vector $\be^T$ is a left-eigenvector of $\bA$ with corresponding eigenvalue $\mu$. 
\end{itemize}
\end{lemma} 

\begin{proof}  
Left-multiply both sides of  (\ref{def:magic:square:row}) by $\be$ to get $\be \be_i^T \bA \be = \be \mu$.  This simplifies to $\bA \be = \mu \be$, hence $\be$ is an eigenvector with associated eigenvalue $\mu$. 

Similarly right-multiply both sides of  (\ref{def:magic:square:col}) by $\be^T$ to get $\be^T \bA \be_j \be^T = \mu \be^T$.  This simplifies to $\be^T \bA = \mu \be^T$, hence $\be^T$ is a left-eigenvector with associated eigenvalue $\mu$. 
\end{proof}

Often, the magic number $\mu$ and vector $\be$ are called the magic eigenvalue and eigenvector.  Mattingly \cite{mattingly:2000} showed some results of eigenvalues and eigenvectors of regular magic squares and later we will generalize the results to other types of magic squares with certain symmetries.  

\subsection{Even-Ordered Regular Magic Squares are Singular}

Trigg \cite{trigg:1948} proved that all order-4, regular (or Dudeney type III) magic squares have zero determinant.\footnote{Although Trigg states that the sequence of integers used to create the magic square take on an arithmetic progression, his proof should extend to any non-natural magic square as defined here.}  Mattingly generalized this to any even-ordered magic square.  We define a regular magic square and reproduce some of Mattingly's notation which is helpful later in this article.    

\begin{definition}
A \textbf{regular magic square} is an order-$n$ magic square $\bA$ with magic number $\mu$ that satisfies 
\begin{align*}
 a_{i,j} + a_{n+1-i,n+1-j} = \frac{2\mu}{n}
\end{align*}
for every $i=1,2,\ldots,n$ and $j=1,2,\ldots,n$. 
\end{definition}

Mattingly \cite{mattingly:2000} noted that a regular magic square $\bA$  satisfies the property 
\begin{align}
	\bA + \bJ \bA \bJ = \frac{2\mu}{n} \bE, \label{eq:regular:perm}
\end{align}
where $\bE=\be \be^T$, the square matrix of all ones.  This is an equivalent definition of a regular magic square, because the matrix $\bJ \bA \bJ$ is the matrix $\bA$ reflected both vertically and horizontally and adding it to $\bA$ results in a constant matrix.  For example, if $\bA$ is the D\"urer magic square in (\ref{durer:magic:square}), then 
\begin{align*}
	\bA + \bJ \bA \bJ  = 
\begin{bmatrix}
16 &	3  &	2  &	13 \\ 
5 &	10 & 	11  &	8 \\ 
9 & 	6  &	7  &	12 \\ 
4  &	15 & 	14  &	1 \\ 
\end{bmatrix} + 
\begin{bmatrix}
1 & 14 & 15 & 4 \\
12 & 7 & 6 & 9 \\
	8 & 11 & 10 & 5 \\
	13 & 2 & 3 & 16 \\
\end{bmatrix} & =
\begin{bmatrix}
	17 & 17 & 17 & 17 \\	17 & 17 & 17 & 17 \\	17 & 17 & 17 & 17 \\	17 & 17 & 17 & 17 \\
\end{bmatrix} = \frac{2\mu}{n} \bE.
\end{align*}
In this case, $\mu=34$ and $n=4$.  

Defining a regular magic square as a matrix $\bA$ that satisfies (\ref{eq:regular:perm}) allows a natural generalization of type III magic squares of  Dudeney in Figure \ref{fig:order4-class} to those of higher order and as we will show below in a similar manner to types I and II as well.  

Mattingly shows two main results.  The first is that if $\lambda$ is an eigenvalue of a regular magic square, $\bA$ with associated eigenvector $\bu$ (and $\lambda \neq \mu$), then $-\lambda$ is also an eigenvalue of $\bA$ with associated eigenvector $\bJ \bu$.  Secondly, he shows that a regular even-ordered magic square is singular.  
We will examine each of these theorems in more detail later in this paper as we generalize these results.

\section{The Role of Permutation Matrices}
\label{sect:perm:matrix}

As we have seen in both the definition of a magic square and a regular magic square,  the reverse matrix $\bJ$ is an important permutation matrix in this context.  Much of this paper explains transformations of magic squares based on permutation matrices, where $\bJ$ is just an example of one.  

In short, a permutation matrix that right (left) multiplies a matrix permutes the rows (columns) of a matrix.  We saw above that if $\bA$ is magic, then $\bJ\bA, \bA\bJ$ and $\bJ\bA\bJ$ is also magic.  We are interested in what permutation matrices $\bP$ result in $\bP\bA, \bA\bP$ and $\bP\bA\bP$ being magic.  

First, consider a square matrix $\bA$ and a permutation matrix $\bP$ of the same order.  The matrix $\bP$ can be considered as a map $P$ where $P_i$ is the column where the 1 appears in the $i$th row.  The map can be considered a set of points of the locations of the 1s.   For example if
\begin{align*}
\bP & = \begin{bmatrix}
0 & 1 & 0 & 0 \\
0 & 0 & 1 & 0 \\
1 & 0 & 0 & 0 \\
0 & 0 & 0 & 1 
\end{bmatrix}
\end{align*}
Then the map is $P=\{ (1,2), (2,3), (3,1), (4,4) \}$ or more compactly can be written (as in \cite{golub:1989}) as $(2~3~1~4)$.    The matrix multiplication $\bA \bP$ is a permutation of the columns of $\bA$ under the map $P$.  The matrix multiplication $\bP\bA$ is a permutation of the rows of $\bA$ under that map $P^{-1}$, which in this example is $(3~1~2~4)$.  Generally the matrix and the map can be used interchangeably, but typically we will denote when each is used.

As we will see permutation matrices play an enormous role in magic squares, but first we will examine their role in semi-magic squares. 

\begin{theorem}
\label{thm:perm:semimagic}
If $\bP$ is a permutation matrix and $\bA$ is a semi-magic square then $\bP\bA, \bA \bP$ and $\bP\bA\bP$ are semi-magic.  
\end{theorem}

\begin{proof}
We first show that $\bP\bA$ is semi-magic by showing that (\ref{def:magic:square:row}) and (\ref{def:magic:square:col}) are satisfied. Let $P$ be the map corresponding to the permutation matrix as defined above.  Equation (\ref{def:magic:square:row}) applied to $\bP\bA$ is 
\begin{align*}
\be_i^T (\bP \bA) \be & = (\bP \be_i)^T \bA \be = \be_{k}^T \bA \be = \mu, 
\end{align*}
where $k=P^{-1}_i$.  This is satisfied for each $k$ in $\{1,2, \ldots, n\}$ since the map $P$ is onto.   Equation (\ref{def:magic:square:col}) applied to $\bP\bA$ is
\begin{align*}
\be^T (\bP \bA) \be_j & = (\be^T \bP) \bA \be_j = \be^T \bA \be_j = \mu.
\end{align*}

The proof that $\bA \bP$ is semi-magic is similar, but not shown here, and the matrix $\bP\bA\bP$ is semi-magic since both $\bP \bA$ and $\bA \bP$ are semi-magic.  

\end{proof}

Throughout this paper, we will use many examples of permutation matrices.  In the order-4 case, there are $4!=24$ such matrices and when enumerated, they are done in a standard manner with $\bP_1=(1~2~3~4), \bP_2=(1~2~4~3), \bP_3=(1~3~2~4), \ldots, \bP_{22}=(4~2~3~1),  \bP_{23}=(4~3~1~2), \bP_{24}=(4~3~2~1)$.  Also, note that $\bP_1=\bI$, the identity matrix and $\bP_{24}=\bJ$.

\subsection{Bisymmetric and 90\textdegree-Symmetric Permutation Matrices}

The permutation matrices that play the most significant role in transformation of magic squares are those that are bisymmetric.  For example, $\bI, \bJ$ as well as
\begin{align} \label{eq:bis:perm:order4}
\bP_{3}&=\begin{bmatrix}
1 & 0 & 0 & 0 \\
0 & 0 & 1 & 0 \\
0 & 1 & 0 & 0 \\
0 & 0 & 0 & 1 
\end{bmatrix} &
\bP_{8} & = \begin{bmatrix}
0 & 1 & 0 & 0 \\
1 & 0 & 0 & 0 \\
0 & 0 & 0 & 1 \\
0 & 0 & 1 & 0 
\end{bmatrix} &
\bP_{17} & = \begin{bmatrix}
0 & 0 & 1 & 0 \\
0 & 0 & 0 & 1 \\
1 & 0 & 0 & 0 \\
0 & 1 & 0 & 0 
\end{bmatrix} &
\bP_{22} &  = \begin{bmatrix}
0 & 0 & 0 & 1 \\
0 & 1 & 0 & 0 \\
0 & 0 & 1 & 0 \\
1 & 0 & 0 & 0 
\end{bmatrix} \end{align}
are the only 6 order-4 bisymmetric permutation matrices.  The subscripts are the standard enumeration of the permutation matrices.  Here we study these matrices including both the count of bisymmetric matrices as well as showing their role in magic squares.

\begin{lemma} \label{lemma:number:bis:perms}
Let $B(n)$ be the number of bisymmetric permutation matrices of order $n$.  Then $B(1)=1, B(2)=2,  B(4)=6$, $B(n)=B(n-1)$ if $n\geq 3$ is odd and $B(n)=2B(n-2)+(n-2)B(n-4)$ if $n\geq 6$ is even.  

\end{lemma}

\begin{proof}
The 1 by 1 matrix $[1]$ is the only bisymmetric permutation matrix of this size and $\bI$ and $\bJ$ are the only order-2 bisymmetric permutation matrices.  In the order-4 case, $\bI, \bJ$ and the four matrices in (\ref{eq:bis:perm:order4}) are the only 6 bisymmetric permutation matrices. 

If $n\geq3$ is odd, then in order to be bisymmetric, there must be a 1 in the center row, center column.  Let $\bP$ be such a matrix.  Then the submatrix $\bP([1,2, \ldots, (n-1)/2,(n+3)/2,\ldots, n)],[1,2, \ldots, (n-1)/2,(n+3)/2,\ldots, n)])$ must also be bisymmetric, and this submatrix is order $n-1$.  Thus, the number of such matrices of odd order must be equal to the number of bisymmetric permutation matrices which is one smaller order or  $B(n-1)$.  

Lastly, if $n \geq 6$ is even, then we construct all matrices of this order.  If the 1 in the first row is in the first column (and thus by symmetry) also in the last row and last column, then the submatrix, $\bP([2,\ldots,n-1],[2,\ldots,n-1])$ must be bisymmetric and the number of these are $B(n-2)$.  An additional $B(n-2)$ bisymmetric permutation matrices are found with a 1 in the first column last row (and again by symmetry) the last row and first column.  

The number of such matrices with the 1 on the first row and in column $j$ with $2\leq j\leq n-1$ is found by noting that this row is $\be_j^T$ and by symmetry the first column is $\be_j$, the last column is $\be_{n+1-j}$ and the last row if $\be_{n+1-j}^T$.  The remaining submatrix to be filled is $(n-4)$ rows by $(n-4)$ columns and must be bisymmetric.  Therefore the number of such matrices is $(n-2)B(n-4)$ and adding to the number of matrices with a 1 in the upper left or upper right corner, the total number of even bisymmetric permutation matrices is $B(n)=2B(n-2)+(n-2)B(n-4)$.  
\end{proof}

We will see later how the knowledge of the number of bisymmetric permutation matrices of a given order helps understand the count of the number of magic squares of a given order.  We now prove a major result between general magic squares and permutation matrices. 

\begin{theorem} \label{thm:bisymm:perm}
	If $\bA$ is a magic square and $\bP$ is a bisymmetric permutation matrix, then $\bP \bA \bP$ is a magic square. 
\end{theorem}

%
%
%

\begin{proof}
Since $\bP$ is bisymmetric, then $\bP=\bP^T$ and $\bJ\bP = (\bJ\bP)^T$ or alternatively, $\bJ \bP = \bP \bJ$, since $\bJ$ is symmetric as well.  Also, from Theorem \ref{thm:perm:semimagic}, $\bP\bA\bP$ is semi-magic.  Thus, we only need to prove  the diagonal sum properties.   Applying (\ref{def:magic:square:diag}) and (\ref{def:magic:square:altdiag}) to $\bP\bA\bP$ results in 
\begin{align*}
\tr(\bP\bA\bP) & =\tr(\bP \bP \bA) = \tr(\bA)=\mu \\
\tr(\bJ \bP\bA\bP) & = \tr(\bP \bJ \bA\bP) = \tr(\bP \bP \bJ \bA) = \tr(\bJ \bA) =\mu
\end{align*}
where $\bJ\bP=\bP\bJ$ is used, for any permutation matrix $\bP^{-1}=\bP^{T}$ and recalling that traces are invariant under cyclic permutations.  (See \cite{lang:1987}). 

\end{proof}

As an example let $\bA$ be the D\"urer Magic square in (\ref{durer:magic:square}) and using the bisymmetric permutation matrix $\bP_3$ given in (\ref{eq:bis:perm:order4}), the result is
\begin{align*}
\bP_3 \bA\bP_3 & = \begin{bmatrix}
16 & 2 & 3 & 13 \\
9 & 7 & 6 & 12 \\
5 & 11 & 10 & 8 \\
4 & 14 & 15 &1 
\end{bmatrix}
\end{align*}
another magic square.  Pre- and post-multiplying $\bA$ by any of the other 5 bisymmetric permutation matrices results in other magic squares (or the same in the case of $\bP=\bI$).  

\begin{lemma} \label{lemma:JP:bisymm}
If $\bP$ is bisymmetric, then $\bJ\bP$ is bisymmetric.
\end{lemma}

\begin{proof}
We need to show that both $(\bJ\bP)=(\bJ\bP)^T$ and $(\bJ\bJ\bP) =(\bJ\bJ\bP)^T$ or $\bP = \bP^T$.  These two statements comprise the definition of bisymmetric. 
\end{proof}


Above, we defined a 90\textdegree-symmetric matrix and gave an example of a 90\textdegree-symmetric permutation matrix in (\ref{eq:rot90:perm:ex}) and here we will show their role in transformations of magic squares and similar to that of bisymmetric permutation matrices, we will find the number of such matrices of a given order.

\begin{lemma} \label{lemma:number:90sym}
Let $R(n)$ be the number of order-$n$ permutation matrices which are 90\textdegree-symmetric. Then 
$R(1)  = 1,  R(2) = 0, R(4) = 2, $
if $n\geq 3$ is odd then $R(n) = R(n-1)$ and if $n$ is even and $n \geq 6$ then $R(n) = (n-2) R(n-4)$.  
\end{lemma}
\begin{proof}
The only order-1 permutation matrix, $[1]$ is also 90\textdegree-symmetric.  There are two order-2 permutation matrices and neither are 90\textdegree-symmetric.  The order-4 case can be shown by construction, let $\bP$ be such a matrix.  The first row cannot have a 1 in the first or last column, because this would imply that there is a 1 in each corner and such a permutation matrix does not exist.  Thus, the first row can have a 1 in the 2nd or 3rd column, and placing the 1 fills in the rest of the matrix.   These two matrices are the matrix in (\ref{eq:rot90:perm:ex}) and $\bJ$ times that matrix.  

Next, consider an odd order 90\textdegree-symmetric permutation matrix, $
\bP$ with $n \geq 3$.  The 1 in the middle row of such a matrix must be in the middle column in order to be 90\textdegree-symmetric.  The submatrix $\bP[(1..(n-1)/2, (n+3)/2..n),(1..(n-1)/2, (n+3)/2..n),]$ must also be a 90\textdegree-symmetric matrix, so $R(n) = R(n-1).$  

Lastly, let $\bP$ be an order-$n$  90\textdegree-symmetric permutation matrix with $n$ even and $n \geq 6$.  This proof continues by construction.  As described above, the 1 in the first row cannot be in the first or last column, so a 1 can be placed in column 2 through $n-1$ (for a total of $n-2$ locations).  Due to the the symmetry required, a one must also be located on the last column, last row and first column and zeros throughout the matrix in all rows and columns where the 1s are located.  This process fills in 4 rows and 4 columns symmetrically and the remaining submatrix (of size $(n-4)$ by $(n-4)$) must be filled in with a  90\textdegree-symmetric permutation matrix.  Thus the number of  90\textdegree-symmetric permutation matrices for an even order is $R(n) = (n-2) R(n-4)$.  
\end{proof} 

This lemma shows that the 90\textdegree-symmetric permutation matrices that exist are doubly-even order (size $4k$) and odd order just larger (size $4k+1$).  We also show that this proof leads to a construction of  90\textdegree-symmetric permutation matrices of a given size.  As an example, let $n=8$.  We can select a 1 in first row (columns 2 through 7), we'll pick the 3rd column.  Since the matrix need to be  90\textdegree-symmetric, this results in 1s in locations to keep the matrix 90\textdegree-symmetric or
\begin{align*}
\begin{bmatrix}
0 & 0 & 1 & 0 & 0 & 0 & 0 & 0 \\
0 &    & 0 &   &    &  0 &    & 0 \\
0 & 0 & 0 & 0 & 0 & 0 & 0 & 1 \\
0 &    & 0 &    &   &  0 &    & 0 \\
0 &     & 0 &   &    & 0 &    & 0 \\
1 & 0 & 0 & 0 & 0 & 0 & 0 & 0 \\
0 &    & 0 &    &    & 0 &    & 0 \\
0 & 0 & 0 & 0 & 0 & 1 & 0 & 0   
\end{bmatrix}
\end{align*}

To fill in the remaining submatrix we can choose either the matrix in (\ref{eq:rot90:perm:ex}) or $\bJ$ times that matrix.  Let's fill in the blank spots above with (\ref{eq:rot90:perm:ex}) to arrive at the 90\textdegree-symmetric permutation matrix:
\begin{align*}
\begin{bmatrix}
0 & 0 & 1 & 0 & 0 & 0 & 0 & 0 \\
0 & 0 & 0 & 1 & 0 & 0 & 0  & 0 \\
0 & 0 & 0 & 0 & 0 & 0 & 0  & 1 \\
0 & 0 & 0 & 0 & 0 & 0 & 1  & 0 \\
0 & 1 & 0 & 0 & 0 & 0 & 0  & 0 \\
1 & 0 & 0 & 0 & 0 & 0 & 0 & 0 \\
0 & 0 & 0 & 0 & 1 & 0 & 0 & 0 \\
0 & 0 & 0 & 0 & 0 & 1 & 0 & 0   
\end{bmatrix}
\end{align*}

The importance of these permutation matrices in the study of magic squares is due to the following theorem. 

\begin{theorem} \label{thm:90sym}
If $\bP$ is a 90\textdegree-symmetric permutation matrix and $\bA$ an order-$n$ magic square, then $\bP\bA\bP$ is a magic square.  
\end{theorem}

\begin{proof}
If $\bP$ is 90\textdegree-symmetric, then it satisfies $\bP^T = \bJ \bP$ or alternatively, $\bP = \bP^T \bJ$.  From Theorem \ref{thm:perm:semimagic}, $\bP\bA\bP$ is semi-magic, so we need to only show the diagonal properties of the magic squares.  Applying (\ref{def:magic:square:diag}) to $\bP\bA\bP$ we get,
\begin{align*}
\tr(\bP\bA\bP) = \tr(\bP\bP\bA) = \tr(\bP \bP^T \bJ \bA) = \tr(\bJ\bA) = \mu
\end{align*}
since traces are invariant under cyclic permutations and for any permutation matrix $\bP^T=\bP^{-1}$.  The last step arises from (\ref{def:magic:square:altdiag}).   In a similar manner, applying (\ref{def:magic:square:altdiag}) to $\bP\bA\bP$ is
\begin{align*}
\tr(\bJ\bP\bA\bP) = \tr(\bP^T \bA \bP) = \tr(\bP \bP^T \bA) = \tr(\bA) =\mu, 
\end{align*}
where the last step is found in (\ref{def:magic:square:diag}).  
\end{proof}

\begin{lemma} \label{lemma:JP:90rot}
If $\bP$ is 90\textdegree-symmetric, then $\bJ\bP$ is 90\textdegree-symmetric. 
\end{lemma}

\begin{proof}
Applying the definition to $\bJ\bP$ leads to $(\bJ\bP)^T =\bJ (\bJ \bP)$ or $(\bJ\bP)^T = \bP$.  Taking the transpose leads to $\bJ \bP = \bP^T$, the definition of 90\textdegree-symmetric. 
\end{proof}

%

\subsection{Extending the Family of Magic Squares} \label{sect:magic:family}

The results of Theorems \ref{thm:bisymm:perm} and  \ref{thm:90sym} indicate that we can extend the number of related magic squares beyond that given by Theorem \ref{thm:8magics}.  From Theorem \ref{thm:8magics} and Lemmas  \ref{lemma:number:bis:perms} and \ref{lemma:number:90sym} there are  $8(B(n)+R(n))$ related magic matrices by applying $\bP_i\bA\bP_i$ to the magic square $\bA$ for $\bP_i$ a bisymmetric or 90\textdegree-symmetric permutation matrix.  However, as we will see  double counting occurs and the number of unique matrices of the form $\bP_i\bA\bP_i$ is only $4(B(n)+R(n))$.  

To examine this more closely, let $\bA$ be an order-4 magic square, and $\bP_i$ is the $i$th standard permutation matrix.  Theorem \ref{thm:bisymm:perm} indicates that $\bI\bA\bI$, $\bP_3\bA\bP_3$ $\bP_8\bA\bP_8$, $\bP_{17}\bA\bP_{17}$, $\bP_{22}\bA\bP_{22}$, and $\bJ\bA\bJ$ are also magic squares.  In addition, $\bP_{11}\bA\bP_{11}$ and $\bP_{14}\bA\bP_{14}$ are also magic from Theorem \ref{thm:90sym}.  If the reflections and rotations in Theorem \ref{thm:8magics} are applied to these eight matrices, there are 64 matrices, howover only 32 are unique.  For example, $\bJ\bA\bJ$ is listed in both Theorem \ref{thm:bisymm:perm} (since $\bJ$ is bisymmetric) and Theorem \ref{thm:8magics}.  Another example shows that if one pre- and post-multiplies  $\bP_{11}\bA\bP_{11}$ by $\bJ$ (from Theorem \ref{thm:8magics}) the result is $\bP_{14}\bA\bP_{14}$, since $\bJ\bP_{11}=\bP_{14}$.    

The double counting occurs because the permutation matrices listed above always come in pairs.  If $\bP$ is either bisymmetric or 90\textdegree-symmetric, then $\bJ\bP$ is as well, as was shown in Lemmas \ref{lemma:JP:bisymm}  and \ref{lemma:JP:90rot}.

We will call the set of matrices starting with magic square $\bA$ and formed by both reflections and rotations as well as transforming using the appropriate permutation matrix the \textbf{family of transformations of $\bA$}.  For example, if we define $\bA$ to be the D\"urer Magic Square in (\ref{durer:magic:square}) then $\bA, \bP_{3}\bA\bP_{3}, \bP_{8}\bA\bP_{8}$ and $\bP_{11}\bA\bP_{11}$ or 

\begin{align*}
\bA & = \begin{bmatrix}
16 & 3 & 2 & 13 \\
5 & 10 & 11 & 8 \\
9 & 6 & 7 & 12 \\
4 & 15 & 14 & 1 
\end{bmatrix} &
\bP_3\bA\bP_3 & = \begin{bmatrix}
16 & 2 & 3 & 13 \\
9 & 7 & 6 & 12 \\
5 & 11 & 10 & 8 \\
4 & 14 & 15 & 1
\end{bmatrix} \\
\bP_8 \bA\bP_8 & = \begin{bmatrix}
10 & 5 & 8 & 11 \\
3 & 16 & 13 & 2 \\
15 & 4 & 1 & 14 \\
6 & 9 & 12 & 7
\end{bmatrix} & 
\bP_{11} \bA \bP_{11} & = \begin{bmatrix}
11 & 5 & 8 & 10 \\
14 & 4 & 1 & 15 \\
2 & 16 & 13 & 3 \\
7 & 9 & 12 & 6
\end{bmatrix}¥
\end{align*}
as well as the 8 rotations and reflections of each is the \emph{family of transformations of the D\"urer magic square}.   In general, the following theorem explains this.

\begin{theorem} \label{thm:family:magics}
Let $P$ be the set of bisymmetric or 90\textdegree-symmetric permutation matrices of order $n$, where only one of $\bP$ or $\bJ\bP$ is included in $P$.  Let $\bA$ be an order-$n$ magic square arising from a set of non-repeating numbers and $\bP_i \in P$, then $\{ \bP_i\bA \bP_i, \bJ \bP_i \bA \bP_i, \bP_i \bA \bP_i\bJ, \bJ\bP_i \bA \bP_i \bJ, \bP_i\bA^T \bP_i, \bJ \bP_i \bA^T \bP_i, \bP_i \bA^T \bP_i\bJ, \bJ\bP_i \bA^T \bP_i \bJ\}$ is a set of unique magic squares.  The size of this set is $\rho(n)=4(B(n)+R(n))$.

\end{theorem}

\begin{proof} For a given $\bP_i$, the 8 matrices listed is the matrix $\bP_i\bA\bP_i$ and the 7 related matrices of Theorem \ref{thm:8magics} formed by reflecting and rotating. Each of these is unique since rotating and reflecting results in a unique matrix.  Also, for a given $\bA$, the matrices $\bP_i\bA\bP_i$ are unique, since each $\bP_i$ is unique.   Lastly, if $\bP_i \in P$ and $\bJ\bP_i \notin P$, then each rotation and reflection of $\bP_i\bA\bP_i$ is unique.

Since only half of the bisymmetric and 90\textdegree-symmetric permutation matrices $\bP_i$ are included in the set $P$, the size of the set for a given order is half of the number of rotations and reflections (eight) times the number of permutation matrices in the set or $B(n)+R(n)$, where these are given in Lemmas \ref{lemma:number:bis:perms} and \ref{lemma:number:90sym}. 

\end{proof}

The following table lists the number of unique magic squares in a family of transformations of a given magic square, for $3 \leq n \leq 10$.  Also, the total number of magic squares for a given order $n$ will be a multiple of this number.  

\begin{center}
 \begin{tabular}{|l|l|l|l|l|l|l|l|l|}\hline
$n$ & 3 & 4 & 5 & 6 & 7 & 8 & 9 & 10 \\ \hline
$4(B(n)+R(n))$ & 8 & 32 & 32 & 80 & 80 & 352 & 352 & 1248 \\ \hline
\end{tabular}
\end{center}

The total number of natural order-3 magic square matrices is 8, so the transformations account for all of them.  In the case of order-4 magic squares, there are 7040 unique matrices, a factor of $\rho(4)=32$, and Richard Schroeppel (given credit posthumously by Martin Gardner \cite{gardner:1976}) found that there are 2,202,441,792 order-5 natural magic square matrices, also a factor of $\rho(5)=32$.  The number of higher-order magic squares are only estimates (see \cite{pinn:1998} for estimates for order 6), but $\rho(n)$ gives a factor of magic squares and would provide a check on the accuracy and the family of magic squares in Theorem \ref{thm:family:magics} can aid in the computation of all magics of a given order in that finding a single magic square results in finding all $\rho(n)$ magic squares is its family.  

The factor of 32 for the order-5 magic squares was shown by Gardner \cite{gardner:1976} in that he states that in addition to the 8 reflections and rotations, there are two other transformations:

\begin{enumerate}[label=\arabic*)]
\item ``Exchange the left and right border columns, then exchange the top and bottom border rows.''
\item ``Exchange rows 1 and 2 and rows 4 and 5.  Then exchange columns 1 and 2, then columns 4 and 5.'' 
\end{enumerate}
And thus he arrives at $4 \cdot 2 \cdot 2 \cdot 2 = 32$ \emph{isomorphic} magic squares---four from the rotations, 2 from the reflections, and 2 each from the two above transformations.  

In light of our work here, the order-5 bisymmetric permutation matrices
\begin{align*}
\bP_{106} & = \begin{bmatrix}
0 & 0 & 0 & 0 & 1 \\
0 & 1 & 0 & 0 & 0 \\
0 & 0 & 1 & 0 & 0 \\
0 & 0 & 0 & 1 & 0 \\
1 & 0 & 0 & 0 & 0
\end{bmatrix} &
\bP_{26} & = \begin{bmatrix}
0 & 1 & 0 & 0 & 0 \\
1 & 0 & 0 & 0 & 0 \\
0 & 0 & 1 & 0 & 0 \\
0 & 0 & 0 & 0 & 1 \\
0 & 0 & 0 & 1 & 0
\end{bmatrix}
\end{align*}
where the numbers are the standard numbering system, are two of the 6 bisymmetric permutation matrices.  If we transform an order-5 magic square $\bA$ as $\bP_{105}\bA\bP_{105}$ and $\bP_{26}\bA\bP_{26}$ the result will be the transforms 1) and 2) listed above.  According to Gardner, these two transformations can be composed, and the other factor of two (in the counting of isomorphic magic squares) results in applying $\bP_{105}\bP_{26}\bA\bP_{26}\bP_{105}$.  The matrix multiplication
\begin{align*}
\bP_{105}\bP_{26} & = 	
\begin{bmatrix}
0 & 1 & 0 & 0 & 0 \\
0 & 0 & 0 & 0 & 1 \\
0 & 0 & 1 & 0 & 0 \\
1 & 0 & 0 & 0 & 0 \\
0 & 0 & 0 & 1 & 0
\end{bmatrix}
\end{align*}
is $\bP_{45}$, a 90\textdegree-symmetric permutation matrix.  Thus our counting of the order-5 isomorphic magic squares is equivalent to Gardner's. 

\subsection{Pandiagonal and Semipandiagonal Magic Squares}

There are two other categories of magic squares that are often studied.  A \textbf{pandiagonal} magic square is one in which all of the \emph{broken} diagonals also add to the magic number.  For example, 
\begin{align}
\begin{bmatrix}
1 & 14 & 4 & 15 \\
8 & 11 & 5 & 10 \\
13 & 2 & 16 & 3 \\
12 & 7 & 9 & 6
\end{bmatrix} \label{ex:pandiagonal} 
\end{align}
is a pandiagonal magic square since it is a magic square (with $\mu=34$) as well as  $a_{j,1}+a_{j+2,2}+a_{j+3,3}+a_{j+4,4}=\mu$ and $a_{j,1}+a_{j+1,4}+a_{j+2,3}+a_{j+3,2}=\mu$ for $j=1,2,3,4$ and addition in the subscripts are performed modulo 4.  (See \cite{pickover:2002} for more details about pandiagonal magic squares.)  Examination of these matrices in the context of permutation matrices are found by the shifts of the identity matrix.  That is, for order-4:
\begin{align*}
\bP_{10} & =
\begin{bmatrix}
0 & 1 & 0 & 0 \\
0 & 0 & 1 & 0 \\
0 & 0 & 0 & 1 \\
1 & 0 & 0 & 0
\end{bmatrix},
&
\bP_{17} & = 
\begin{bmatrix}
0 & 0 & 1 & 0 \\
0 & 0 & 0 & 1 \\
1 & 0 & 0 & 0 \\
0 & 1 & 0 & 0
\end{bmatrix},
&
\bP_{19} & = 
\begin{bmatrix}
0 & 0 & 0 & 1 \\
1 & 0 & 0 & 0 \\
0 & 1 & 0 & 0 \\
0 & 0 & 1 & 0
\end{bmatrix}.
\end{align*}
We can then use these permutation matrices to shift the broken diagonals to the main diagonals.  For example, if $\bA$ is the magic square in (\ref{ex:pandiagonal}) then  
\begin{align*}
\bA\bP_{10} & = \begin{bmatrix}
15 & 1 & 14 & 4 \\
10 & 8 & 11 & 5 \\
3 & 13 & 2 & 16 \\
6 & 12 & 7 & 9
\end{bmatrix},
&
\bA \bP_{10} \bJ & = \begin{bmatrix}.
4 & 14 & 1 & 15 \\
5 & 11 & 8 & 10 \\
16 & 2 & 13 & 3 \\
9 & 7 & 12 & 6
\end{bmatrix}.
\end{align*}
The matrix $\bA\bP_{10}$ has shifted the broken diagonal below the main diagonal to the main diagonal and $\bA\bP_{10}\bJ$ has shifted the broken diagonal above the minor diagonal to the main diagonal.   When transformed as $\bA\bP$ and $\bA\bP\bJ$, the permutation matrices $\bP_{17}$ and $\bP_{19}$ shift the other broken diagonal to the main diagonal.  

In this light, we can formally define a pandiagonal magic square using permutation matrices.  

\begin{definition}
If $\bA$ is a magic square with magic number $\mu$, then $\bA$ is \textbf{pandiagonal} if $\bA$ satisfies both 
\begin{align*}
\tr(\bA\bP)& = \mu & & \text{and} &  \tr(\bA\bP\bJ)& = \mu,
\end{align*}
for all $\bP$ which are shifts of the identity matrix.  
\end{definition}
Although the examples above used only order-4 matrices, the formal definition can be extended to any order.  Pandiagonal magic squares of higher order are often studied.  

Another often-used classifying criteria is that of a semipandiagonal or semi-Nasik magic square.  In this case only two of the broken diagonals are used---in the order 4 case only the one corresponding to $\bP_{17}$ and $\bP_{17}\bJ$.   Again, using permutation matrices, we can now extend the definition more formally. 
\begin{definition}
If $\bA$ is a magic square with magic number $\mu$, then $\bA$ is \textbf{semipandiagonal} if 
\begin{align*}
\tr(\bA\bP) &=\mu &&\text{and}& \tr(\bA\bP\bJ) & = \mu
\end{align*}
where $\bP$ is the shift of the identity matrix with $\bP$ symmetric.  
\end{definition}
  
This definition allows the generalization of semi-pandiagonal magic squares to orders higher than 4.  There is no permutation matrix of order 5 which is a shift of the identity matrix and symmetric.  For order 6, the permutation matrix that play the important role in the definition is:
\begin{align*}
\begin{bmatrix}
0 & 0 & 0 & 1 & 0 & 0 \\
0 & 0 & 0 & 0 & 1 & 0 \\
0 & 0 & 0 & 0 & 0 & 1 \\
1 & 0 & 0 & 0 & 0 & 0 \\
0 & 1 & 0 & 0 & 0 & 0 \\
0 & 0 & 1 & 0 & 0 & 0 
\end{bmatrix}, 
\end{align*}
and an example of an order-6 semipandiagonal magic square is
\begin{align*}
\begin{bmatrix}
19 & 18 & 7 & 14 & 29 & 33 \\
27 & 8 & 26 & 4 & 34 & 21 \\
23 & 31 & 36 & 15 & 13 & 2 \\
39 & 6 & 5 & 17 & 25 & 28 \\
1 & 22 & 37 & 32 & 16 & 12 \\
11 & 35 & 9 & 38 & 3 & 24
\end{bmatrix}
\end{align*} 
where $\mu=120$.  

\section{Classifying Dudeney Types using Permutation Matrices} \label{sect:dudeney}

Next, we examine the work of Dudeney \cite{dudeney:1917} and his twelve order-4 types of magic squares which is shown in Figure \ref{fig:order4-class}.  We first investigate a simple question for all order-4 magic squares: if $\bP$ is a permutation matrix, and $\bA_q$ is a magic square of type $q$, for which permutation matrices is $\bP\bA_q\bP$ a magic square and what type is the result?  In investigating this, we are able to generalize types I--VI to higher order magic squares and show properties of their determinants and eigenspectra.

The directed graphs in Figure \ref{fig:order4:trans} summarize the results of transforming all order-4 magic squares under $\bP\bA\bP$ for a permutation matrix $\bP$.  The upper-case letter types (A--D) listed above each graph are the same as the classification of Trigg \cite{trigg:1948} who grouped Dudeney's roman numeral types.  Within each of the four graphs, the labelled arrows denote which sets of permutation matrices transform to other magic squares. The vertices of each graph are the Dudeney type.  For example within type A, permutation matrices in the set $A_2$ transform type I magic squares to type III magic squares (and vice versa).   

\begin{figure}

\begin{tabular}{p{3.25in}p{3.25in}}
\begin{minipage}{3in}
\begin{center}
{\large \textbf{type A}}

\psset{arcangle=12,arrowsize=0.2,nodesep=3pt}
\begin{pspicture}(-2,-2)(6,4.75)
\cnodeput{0}(2,3){A}{I}
\cnodeput{0}(0,0){B}{II}
\cnodeput{0}(4,0){C}{III}
\ncline{<->}{A}{B} \tlput{$A_4$}
\ncline{<->}{A}{C} \trput{$A_2$}
\ncline{<->}{C}{B} \tbput{$A_3$}
\nccircle{->}{A}{.4}
\nccircle[angleA=135]{->}{B}{.4}
\nccircle[angleA=225]{->}{C}{.4}
\uput[90](2,4){$A_1,A_3$}
\uput[180](-0.7,-0.6){$A_1,A_2$} 
\uput[0](4.7,-0.6){$A_1,A_4$} 
\end{pspicture}
\end{center}
\end{minipage}
& 
\begin{minipage}{3in}
\begin{center}
{\large \textbf{type B}}
 \psset{arcangle=12,arrowsize=0.2,nodesep=3pt}
\begin{pspicture}(-2,-2)(6,4.75) 
\cnodeput{0}(0,3){A}{IV}
\cnodeput{0}(0,0){B}{V}
\cnodeput{0}(3,0){C}{VI$'$}
\cnodeput{0}(3,3){D}{VI$''$}
\ncline{<->}{A}{B} \tlput{$A_4$}
\ncline{<->}{A}{C} \trput[tpos=0.6]{$A_3$}
\ncline{<->}{C}{B} \tbput{$A_2$}
\ncline{->}{D}{B} \tbput[tpos=0.45]{$A_2$}
\ncline{->}{D}{A} \taput{$A_3$}
\nccircle{->}{A}{.4}
\nccircle[angleA=135]{->}{B}{.4}
\nccircle[angleA=225]{->}{C}{.4}
\nccircle[angleA=300]{->}{D}{.4}
\uput[90](0,4){$A_1,A_2$}
\uput[180](-0.7,-0.6){$A_1,A_3$} 
\uput[0](3.7,-0.6){$A_1,A_4$} 
\uput[45](3,3.75){$A_1,A_4$} 
\end{pspicture}
\end{center}
\end{minipage}
 \\
\begin{minipage}{3in	}
\begin{center}
{\large \textbf{type C}}

\psset{arcangle=12,arrowsize=0.2,nodesep=3pt}
\begin{pspicture}(-2,-1)(6,4.5) 
\cnodeput{0}(0,0){A}{X}
\cnodeput{0}(0,3){B}{VII}
\cnodeput{0}(4,3){C}{VIII}
\cnodeput{0}(4,0){D}{IX}
\ncline{<->}{A}{B} \tlput{$C_2$}
\ncline{<->}{A}{C} \trput{$C_3$}
\ncline{<->}{C}{B} \taput{$C_4$}
\ncline{<->}{A}{D} \tbput{$C_4$}
\ncline{<->}{C}{D} \trput{$C_2$}
\ncline{<->}{B}{D}\tlput{$C_3$}
\nccircle[angleA=135]{->}{A}{.4}
\nccircle[angleA=45]{->}{B}{.4}
\nccircle[angleA=-45]{->}{C}{.4}
\nccircle[angleA=-135]{->}{D}{.4}
\uput[-90](-1.2,3.5){$C_1$}
\uput[180](-0.7,-0.6){$C_1$} 
\uput[0](4.7,-0.6){$C_1$} 
\uput[90](5.2,3.0){$C_1$}
\end{pspicture}

\end{center}
\end{minipage}
&
\begin{minipage}{3in	}
\begin{center}
{\large \textbf{type D}}

\psset{arcangle=12,arrowsize=0.2,nodesep=3pt}
\begin{pspicture}(-2,-1)(6,4.5) 
\cnodeput{0}(0,1.5){A}{XI}
\cnodeput{0}(3,1.5){B}{XII}
\ncline{<->}{A}{B} \taput{$A_4$}
\nccircle[angleA=90]{->}{A}{.4}
\nccircle[angleA=-90]{->}{B}{.4}
\uput[180](-1,1.5){$A_1$}
\uput[0](4,1.5){$A_1$} 
\end{pspicture}

\end{center}
\end{minipage}
\end{tabular}

\begin{caption}
A summary of the transformation of order-4 magic squares.  If $\bA$ is magic, and $\bP$ is a permutation matrix, then the diagrams show for which permutation matrices is $\bP\bA\bP$ magic and the Dudeney type of the original and transformed matrix.  The vertices of the graphs are the Dudeney types and the arrows are labelled with the transformation sets:  $A_1 = \{\bP_1, \bP_8,\bP_{17},\bP_{24}\},  A_2  = \{\bP_2, \bP_7,\bP_{18},\bP_{23} \}, A_3  = \{\bP_6,\bP_{10},\bP_{15},\bP_{19}\},   A_4 = \{\bP_3,\bP_{11},\bP_{14},\bP_{22} \}, C_1  = \{\bP_{1},\bP_{24}\},  C_2  =\{\bP_{3},\bP_{22}\},  C_3 = \{\bP_{8},\bP_{17}\},  C_4  = \{\bP_{11},\bP_{14}\}$, where the subscripts are the index of the permutation matrices in standard form.  The types (A--D) listed above each graph correspond to the types that Trigg \cite{trigg:1948}  defined in terms of Dudeney's roman numerical types.  Also, the set VI$'$ is the type VI that are also semipandiagonal, whereas VI$''$ is not semipandiagonal.  \end{caption}
\label{fig:order4:trans}
\end{figure}

As can be seen in Figure \ref{fig:order4:trans}, types I-III (type A) are always transformed within that set, as are types IV--VI (type B), VII-X (type C) and lastly the set within XI and XII (type D).  As was shown in Theorems \ref{thm:bisymm:perm} and \ref{thm:90sym}, the set of bisymmetric and 90\textdegree-symmetric permutation matrices 
transform magic squares to other magic squares.  This set is either $A_1 \bigcup A_4$ or $C_1 \bigcup C_2 \bigcup C_3 \bigcup C_4$.  

In the case of types I--VI, there are additional permutation matrices (the set $A_2 \bigcup A_3$) that transform some magic squares to other squares.   In addition, type VI magic squares do not transform in a consistent manner and thus to be fully understood has be subdivided into two sets VI$'$, which are type VI magic squares that are semipandiagonal and VI$''$, which do not have the property.  

\subsection{Magic Classifying Permutation Matrices} 

The transformations of type A and B magic squares can be generalized to higher-order magic squares.  We will show that the generalized type A (and type B) magic squares can be transformed into other type A (and type B) magic squares using other permutation matrices.  We first start defining the type I and II order-4 magic squares in a manner similar to that of Mattingly \cite{mattingly:2000}. If 
\begin{align} \label{def:LandK}
	\bK & = 
\begin{bmatrix}
	0 & 0 & 1 & 0 \\ 0 & 0 & 0 & 1 \\
	1 & 0 & 0 & 0 \\ 0 & 1 & 0 & 0 
\end{bmatrix} & &\text{and}& 
\bL& = 
\begin{bmatrix}
	0 & 1 & 0 & 0 \\ 1 & 0 & 0 & 0 \\ 0 & 0 & 0 & 1 \\ 0 & 0 & 1 & 0 
\end{bmatrix}
\end{align}
then if $\bA_I$ is a type I magic square then $\bA_I$ satisfies
\begin{align}
	\bA_I + \bK \bA_I \bK =\frac{2\mu}{n}\bE , \label{eq:typeI:perm}
\end{align}
and if $\bA_{II}$ is a type II magic square then $\bA_{II}$ satisfies
\begin{align}
	\bA_{II} + \bL \bA_{II} \bL =\frac{2\mu}{n}\bE . \label{eq:typeII:perm}
\end{align}
Equations (\ref{eq:typeI:perm}) and (\ref{eq:typeII:perm}) are analogs of  (\ref{eq:regular:perm}) for types I and II.  We illustrate the relationship between the diagrams of Dudeney with this equation through an example.  The matrix 
\begin{align*}
 \bA = 
\begin{bmatrix}
 1 & 14 & 4 & 15 \\
8 & 11 & 5 & 10 \\
13 & 2 & 16 & 3 \\
12 & 7 & 9 & 6
\end{bmatrix}
\end{align*}
is a natural order-4 magic square ($\mu=34$) that is type I.  If we substitute this into (\ref{eq:typeI:perm}),  
\begin{align*}
 \bA + \bK \bA \bK & = 
\begin{bmatrix}
 1 & 14 & 4 & 15 \\
8 & 11 & 5 & 10 \\
13 & 2 & 16 & 3 \\
12 & 7 & 9 & 6
\end{bmatrix} + \begin{bmatrix}
16 & 3 & 13 & 2 \\
9 & 6 & 12 & 7 \\
4 & 15 & 1 & 14 \\
5 & 10 & 8 & 11 \end{bmatrix} = 
\begin{bmatrix}
 17 & 17 & 17 & 17 \\
  17 & 17 & 17 & 17 \\
   17 & 17 & 17 & 17 \\
    17 & 17 & 17 & 17 \\
\end{bmatrix} = \frac{2\mu}{n} \bE,
\end{align*}
where $\mu=34$ and $n=4$, so the type I matrix $\bA$ satisfies (\ref{eq:typeI:perm}).  

Magic squares of type IV, V,  VI, XI and XII can also be written in terms of permutation matrices.  Let $\bA_q$ be a magic square of type $q$, then the following are true: 
\begin{align}
	\bA_{IV} +\bL \bA_{IV} & = \frac{2\mu}{n} \bE, \label{eq:typeIV:perm}\\
	\bA_{V} + \bK \bA_{V} & =  \frac{2\mu}{n} \bE, \label{eq:typeV:perm}\\
	\bA_{VI} + \bJ \bA_{VI} & =  \frac{2\mu}{n} \bE, \label{eq:typeVI:perm}\\
	\bA_{XI} + \bL \bA_{XI} \bP_3
& =  \frac{2\mu}{n}\bE, \label{eq:typeXI:perm} \\
	\bA_{XII} + \bK \bA_{XII} \bP_2
& =  \frac{2\mu}{n} \bE. \label{eq:typeXII:perm}
\end{align}
where $\bP_3$ is defined in (\ref{eq:bis:perm:order4}) and $\bP_2=(1~2~4~3)$, written in compact notation.   

Initially, we examine types I--III  which are written as equations using matrix operations in  (\ref{eq:regular:perm}), (\ref{eq:typeI:perm}) and (\ref{eq:typeII:perm}) and note that $\bK, \bL$ and $\bJ$ play an important role in these equations.  Those matrices will be called magic classification permutation matrices in part because these permutation matrices can be used to classify a magic square in a way similar to that of Dudeney.

  \begin{definition}\label{def:mcpm}
	A \textbf{magic classifying permutation matrix} or \textbf{MCPM} is a symmetric permutation matrix, $\bP$ of order $n$ and 
	
\begin{itemize}
 \item if $n$ is even then  $\tr(\bP)=0$. 
 \item if $n$ is odd then the middle element of $\bP$ is 1 and $\tr(\bP)=1$. 
\end{itemize}

\end{definition}

The matrix $\bJ$ is the only MCPM of order 2 and 3 and the matrices $\bJ, \bL$ and $\bK$ are the only magic classification permutation matrices of order $4$.  We note that $\bK=\bP_{17}$ and $\bL=\bP_{8}$ as shown in (\ref{eq:bis:perm:order4}) and the other two bisymmetric permutation matrices in (\ref{eq:bis:perm:order4}), $\bP_3$ and $\bP_{22}$ are not MCPMs, nor are $\bP_{11}$ (seen in (\ref{eq:rot90:perm:ex}) and $\bP_{14}$, the two 90\textdegree-symmetric permutation matrices.

Before showing how to generalize the types I--III order-4 magic squares, it will be helpful to know how many MCPMs there are of a given order.

\begin{lemma} \label{lemma:pn}
Let $C(n)$ denote the number of MCPMs of order $n\geq 2$, then
\begin{align*}
 C(n) & = 
\begin{cases}
 (n-1)!! & \text{if $n$ is even}, \\
 (n-2)!! & \text{if $n$ is odd}. 
\end{cases}
\end{align*}
\end{lemma}
\begin{proof}
We will prove this result by constructing even-order MCPMs, then showing the relationship between even and odd order.   

The only MCPM of order 2 is $\bJ$, so $C(2)=1$. Let $\bP$ be a MCPM of order $n\geq 4$ and let $j$ satisfy $2 \leq  j \leq n$ and be the column of the 1 in the first row of the matrix $\bP$;  row 1 is $\be_j^T$ and column $j$ is $\be_1$.  (Note that to satisfy $\tr(\bP)=0$, the 1 cannot lie in column 1.)    Also, by symmetry column 1 is $\be_j$ and the $j$th row is  $\be_1^T$. The remainder of the matrix is the submatrix $\bP[(2..j-1,j+1..n),(2..j-1,j+1..n)]$ and must be a MCPM of order $n-2$, which can be constructed in the same way. 

Thus the number of MCPMs of order $n$ is the number of choices for placing the 1 in the first row times the number of MCPMs of order $n-2$ or $C(n)=(n-1) \cdot C(n-2)$.  Since $C(2)=1$, for $n$ an even number $C(n)=(n-1)!!$. 

If $n$ is odd, the middle term must be a 1 by definition, and the submatrix resulting from removing the middle column and middle row is an order-$(n-1)$ MCPM.  Thus the number of odd magic permutation matrices is the same as that of the even-ordered matrix just smaller than it or $(n-2)!!$.  
\end{proof}

As an example, we show how to create a MCPM of order 6.  We start with a 1 in the 4th column of the first row (and could have chosen any column between 2 and 6) and due to the fact that a MCPM is symmetric and a permutation matrix, so far this looks like:
\begin{align*} 
\begin{bmatrix}
 0 &0  & 0& 1 & 0& 0\\
 0 & & & 0 & & \\
 0 & & & 0 & & \\
 1 & 0 & 0 & 0 & 0 & 0 \\
 0 &  & & 0 & & \\
 0 & & & 0 & & 
\end{bmatrix}
\end{align*}

The remaining submatrix consisting of the empty matrix elements above, must also be a MCPM.  There are three choices for this: $\bK, \bL,$ or $\bJ$.  To continue the example, consider $\bL$ as the submatrix and the result is: 
\begin{align}  \label{eq:ex:order6:MCPM} 
\bP & = \begin{bmatrix}
 0 & 0 & 0 & 1 & 0 & 0 \\
 0 & 0 & 1 & 0 & 0 & 0 \\
 0 & 1 & 0 & 0 & 0 & 0 \\
 1 & 0 & 0 & 0 & 0 & 0 \\
 0 & 0 & 0 & 0 & 0 & 1 \\
 0 & 0 & 0 & 0 & 1 & 0 
 \end{bmatrix}
\end{align}
This order-6 example is one of $5!!$ or 15 MCPMs.  The first step in creating them is to place a 1 in one of the 2nd through 6th columns of the first row, thus 5 choices.  By symmetry, there is a 1 in the same row of the first column.  Then the resulting 4 by 4 submatrix can be either $\bK, \bL$ or $\bJ$.


\subsection{Classification of Magic Squares by Trigg Type}

Trigg \cite{trigg:1948} found the determinant of each of the order-4 880  magic squares (or 7040 unique matrices) based on categories of Dudeney types.  In his proof, he classifies all 4 by 4 magics into four types (labelled A through D) as discussed above.  Trigg's proof involves operations to change types II and III into a type I magic square.  Precisely, he says for a type II magic: \emph{Interchange 2nd and 3rd rows, and then interchange 2nd and 3rd columns to get Type I.}  To show how this language is equivalent to using permutation matrices, recall that a type II magic square satisfies (\ref{eq:typeII:perm}) and let $\bA_{II}$ be a magic square of type II.  Interchanging the 2nd and 3rd rows as well as the 2nd and 3rd columns of $\bA_{II}$, a type II matrix, is $
	\bA = \bP_3 \bA_{II} \bP_3$
where $\bP_3$ is defined in (\ref{eq:bis:perm:order4}). When (\ref{eq:typeII:perm}) is multiplied both on the left and right by $\bP_3$, the result is 
\begin{align*}
	\bP_3 \bA_{II}\bP_3 + \bP_3 \bL \bA_{II} \bL \bP_3 = \frac{2\mu}{n} \bP_3 \bE \bP_3 
\end{align*}
and using that $\bP_3\bL=\bK\bP_3$, $\bL\bP_3=\bP_3\bK$  as well as $\bP_3 \bE \bP_3 = \bE$ to get
\begin{align*}
	\bP_3 \bA_{II}\bP_3 + \bK \bP_3 \bA_{II} \bP_3 \bK = \frac{2\mu}{n} \bE. 	
\end{align*}
This shows that the matrix $\bA=\bP_3\bA_{II}\bP_3$ is a type I matrix from (\ref{eq:typeI:perm}).  

The Trigg classification can be linked to the MCPM by the following definition. 

\begin{definition} 
If $\bA$ is a magic square that satisfies 
\begin{align} \label{eq:typeA}
 \bA + \bP \bA \bP & = \frac{2\mu}{n} \bE 
\end{align}
for some MCPM $\bP$, then $\bA$ is called a \textbf{type A magic square}.  If the particular permutation matrix $\bP$ is important, then it is called a \textbf{type $\mbox{A}_{\bP}$ magic square}.
\end{definition}

In this language, the order-4 type I, II, and III magic squares can respectively be called type $\mbox{A}_{\bK}$, $\mbox{A}_{\bL}$ and $\mbox{A}_{\bJ}$ magic squares and any regular magic square is of type $\mbox{A}_{\bJ}$.   This definition also clearly allows the categorization of magic squares to that of any order.   This and the next definition which generalizes type B magics is why we have named this set of permutation matrices, \emph{magic classifying permutation matrices}.

As an example of an order-6 type A magic square, we use a magic square (similar to that of Mattingly) consisting of the set $\{1,2,3,\ldots,38,39\}$ in which the numbers 10, 20 and 30 have been removed.  (A discussion of why we have done this is below.)   The following magic square:
\begin{align}
\bA = \begin{bmatrix}
14 &   33 &    34  &    5 &    12 &    22 \\
19 &    27 &    39 &     8 &    25 &     2 \\
32 &     1 &    13 &    21 &    38 &    15 \\
35 &     6 &     7 &    26 &    18 &    28 \\
11 &    36 &    23 &    31 &     3 &    16 \\
  9 &    17 &     4 &    29  &   24 &    37\\
\end{bmatrix} \label{ex:typeA:6by6}
\end{align}
is a type $\mbox{A}_{\bP}$, where $\bP$ is the MCPM in (\ref{eq:ex:order6:MCPM}) in that it satsifies $\bA+\bP\bA\bP = (2\mu/n)\bE$ where $\mu=120$ and $n=6$.  



%
%
%

\begin{definition}
If $\bA$ is a magic square that satisfies either of the following:
	\begin{align}
 \bA + \bP \bA &  = \frac{2 \mu}{n}\bE,  \label{def:typeBi} \\
  \bA + \bA \bP & = \frac{2 \mu}{n}\bE,  \label{def:typeBii}
\end{align}
for some MCPM, $\bP$ then $\bA$ is called a \textbf{type B magic square}.  If the specific permutation matrix is important, we call the matrix a \textbf{type $\mbox{B}_{\bP}$ magic square}.  
\end{definition}

It is noted that types IV--VI satisfy equation (\ref{def:typeBi}) and on the surface (\ref{def:typeBii}) is unnecessary, however it is needed if magic squares are treated as matrices.  Recall that  Dudeney and Trigg assumed that a magic square that satisfies (\ref{def:typeBii}) could be rotated to one that satisfies (\ref{def:typeBi}).  In a similar manner to that mentioned above, the order-4 magic squares of types IV, V and VI can be called type $\mbox{B}_{\bL}$, $\mbox{B}_{\bK}$ and $\mbox{B}_{\bJ}$.  

Also, as an example of a type B order-6 magic square, 
\begin{align}
\bA & = \begin{bmatrix}
9 & 12 & 28 & 31 & 26 & 14 \\
24 & 36 & 4 & 16 & 21 & 19 \\
32 & 3 & 37 & 8 & 22 & 18 \\
11 & 27 & 13 & 29 & 6 & 34 \\
5 & 25 & 15 & 35 & 7 & 33 \\
39 & 17 & 23 & 1 & 38 & 2
\end{bmatrix}\label{ex:typeB:6by6}
\end{align}
satisfies (\ref{def:typeBii}), where $\bP$ is found in (\ref{eq:ex:order6:MCPM}) and $\mu=120$.  This matrix is singular, has rank 4 and later we will show the rank and determinant of all type B magics.

\subsection{Type A Magic Squares}

In Theorems \ref{thm:bisymm:perm} and \ref{thm:90sym}, we showed that if $\bA$ is a magic square and $\bP$ is either bisymmetric or 90\textdegree-symmetric that   $\bP\bA\bP$ is magic.  This next theorem generalizes some of the results seen in the order-4 examples shown above.

\newcommand{\typeA}[1]{\ensuremath \mbox{A}_{#1}}

%
%

\begin{theorem} \label{thm:convert:magic}
 If $\bQ$ is a bisymmetric permutation matrix, $\bA$ is a type $\mbox{A}_{\bP}$ magic square, and $\bP'=\bQ\bP\bQ$ then $\bQ\bA\bQ$ is a type $\mbox{A}_{\bP'}$ magic square.  
\end{theorem}

\begin{proof}
The matrix $\bQ\bA\bQ$ is a magic square due to Theorem \ref{thm:bisymm:perm}. Also, since $\bA$ is a type $\mbox{A}_{\bP}$ magic square, it satisfies (\ref{eq:typeA}).  Multiplying (\ref{eq:typeA}) on both the left and right by $\bQ$ results in
\begin{align*}
 \bQ \bA \bQ + \bQ \bP \bA \bP \bQ & = \frac{2\mu}{n} \bE
\end{align*}
Letting $\bP=\bQ\bP'\bQ$, and using $\bQ\bQ=\bI$ and $\bQ\bE\bQ=\bE$, then 
\begin{align*}
 \bQ \bA \bQ + \bQ \bQ\bP'\bQ \bA\bQ \bP \bQ\bQ   & =  \bQ\bA\bQ+\bP'\bQ\bA\bQ \bP' =\frac{2\mu}{n} \bE
\end{align*}
which shows that $\bQ\bA\bQ$ is a type $\typeA{\bP'}$ magic square.  

Lastly, we need to show that $\bP'$ is a MCPM, that is symmetric and has the same trace as $\bP$.  First,  $\bP'^{T}=(\bQ \bP \bQ)^{T} = \bQ^{T} \bP^{T} \bQ^{T} =\bQ \bP \bQ=\bP'$ and $\tr(\bP')=\tr(\bQ\bP\bQ)=\tr(\bQ\bQ\bP)=\tr(\bP)$ since traces are invariant under cyclic permutations.  Thus $\bP'$ is a MCPM.   

\end{proof}

In Figure \ref{fig:order4:trans}, we presented a set of directed graphs representing the transformations of order-4 magic squares by pre- and post-multiplication by permutation matrices.  For example, a type I (or type $\typeA{\bK}$) matrix is transformed to a type II (or type $\typeA{\bL}$) by $\bP_3$.  Using Theorem \ref{thm:convert:magic}, we let $\bP=\bK$ in (\ref{def:LandK}) and $\bQ=\bP_3$ in (\ref{eq:bis:perm:order4}), then $\bP'=\bQ\bP\bQ=\bP_3\bK\bP_3= \bL$, therefore using $\bP_3$, the type $\typeA{\bK}$ (type I) matrix is transformed to a type $\typeA{\bL}$ (type II) matrix.  


There are permutation matrices other than bisymmetric ones that transform type A magic squares to other type A magic squares.  For example, $\bP_2$ (which is symmetric, but not bisymmetric) transforms a type $\typeA{\bK}$ (type I) to a $\typeA{\bJ}$ (type III)---see Figure \ref{fig:order4:trans}.  From the results in Figure \ref{fig:order4:trans}, the set of permutation matrices that transform type A magic squares have the property that the permutation matrix is singly symmetric.  However, moving to higher-order magic squares, this is not true.   It appears from examining type A magic squares with sizes above 4 by 4, that further analysis is needed.  

The following result shows that any type $\typeA{\bP}$  magic square can be transformed to any other type $\typeA{\bP'}$ magic square.

\begin{theorem} \label{thm:convert:MCPMs}
If $\bP$ and $\bP'$ are even-ordered MCPMs, then there exists a symmetric permutation matrix $\bQ$, such that $\bP' = \bQ \bP\bQ$.
\end{theorem}

\begin{proof}
Let the size of the matrices be $2n$ and ${\cal P}$ and ${\cal P}'$ be the graphs that correspond to the permutation matrices $\bP$ and $\bP'$.  Since $\bP$ and $\bP'$ are even-ordered MCPMs, the graphs ${\cal P}$ and ${\cal P}'$ consist of $n$ edges, each which connects a vertex of degree 1.  These two graphs are isomorphic and from Goodaire and Parmenter \cite{goodaire:2006}, there exists a permutation matrix $\bQ$ such that $\bP=\bQ\bP'\bQ^T$, where $\bQ$ is the permutation matrix that relabels the vertices of ${\cal P}$ to that of ${\cal P'}$.  Because of the structure of these graphs, the relabeling is done in a symmetric way---that is, the relabeling occurs as swaps. Thus the permutation matrix $\bQ$ is symmetric and the theorem holds.  

\end{proof}

This theorem gives further insight into the categorization and properties of magic squares. For example, since there exists a matrix that converts a type $\typeA{\bP}$ magic square to a type $\typeA{\bP'}$ magic square, this implies that there are equal numbers of magic squares of $\bP_i$-type for every MCPM, $\bP_i$.   This is because for each type $\typeA{\bP}$ magic square, there is a transformation to a type $\typeA{\bP'}$ magic square.   In the order-4 case, there are 384 matrices of each type I, II and III (of Dudeney's classification).      

Theorem \ref{thm:convert:MCPMs} also shows how to select $\bQ$ to convert from a type $\typeA{\bP}$ magic square to a type $\typeA{\bP'}$ magic square.  Consider an example where 
\begin{align}
 \bP & = 
\begin{bmatrix}
 0&1&0&0&0&0\\ 1&0&0&0
&0&0\\ 0&0&0&0&1&0\\ 0&0&0&0&0&1
\\ 0&0&1&0&0&0\\ 0&0&0&1&0&0
\end{bmatrix}, & 
\bP' &=
\begin{bmatrix}
0&0&0&1&0&0\\ 0&0&0&0
&0&1\\ 0&0&0&0&1&0\\ 1&0&0&0&0&0
\\ 0&0&1&0&0&0\\ 0&1&0&0&0&0 
\end{bmatrix}.
\label{eq:convert:example}
\end{align}

We build the following graphs ${\cal P}$ (on the left) and ${\cal P}'$ (on the right): 
\begin{multicols}{2}
\begin{center}
\begin{pspicture}(0,-0.5)(5,4)
\psdots(1,0)(3,0)(4,1.5)(3,3)(1,3)(0,1.5)
\uput[135](1,3){1}
\uput[45](3,3){2}
\uput[0](4,1.5){3}
\uput[-45](3,0){4}
\uput[-135](1,0){5}
\uput[180](0,1.5){6}
\psline(1,3)(3,3)
\psline(1,0)(4,1.5)
\psline(3,0)(0,1.5)
\end{pspicture}

\end{center} \columnbreak
\begin{center}
\begin{pspicture}(0,-0.5)(5,4)
\psdots(1,0)(3,0)(4,1.5)(3,3)(1,3)(0,1.5)
\uput[135](1,3){1}
\uput[45](3,3){2}
\uput[0](4,1.5){3}
\uput[-45](3,0){4}
\uput[-135](1,0){5}
\uput[180](0,1.5){6}
\psline(1,3)(3,0)
\psline(3,3)(0,1.5)
\psline(4,1.5)(1,0)
\end{pspicture}

\end{center}
\end{multicols}

The relabeling of the first graph as swapping vertices 2 and 4 results in the second graph.  Thus the permutation matrix 
\begin{align} \bQ & = 
\begin{bmatrix}
1 & 0 & 0 & 0 & 0 & 0 \\
0 & 0 & 0 & 1 & 0 & 0 \\
0 & 0 & 1 & 0 & 0 & 0 \\
0 & 1 & 0 & 0 & 0 & 0 \\
0 & 0 & 0 & 0 & 1 & 0 \\
0 & 0 & 0 & 0 & 0 & 1 
\end{bmatrix} \label{ex:Q:transform}
\end{align}
performs the relabeling and this matrix satisfies $\bP = \bQ\bP'\bQ$ for $\bP$ and $\bP'$ in (\ref{eq:convert:example}).  There are also 3 additional swaps and permutation matrices that will also convert $\bP$ to $\bP'$, thus $\bQ$ is not unique.

%
%
%
%
%

The matrix 
\begin{align*}
 \bA & = 
\begin{bmatrix}
 32 & 27 & 4 & 16 & 39 & 2 \\
13 & 8 & 1 & 38 & 36 & 24 \\
23 & 14 & 31 & 7 & 11 & 34 \\
5 & 19 & 37 & 25 & 22 & 12 \\
26 & 17 & 29 & 6 & 9 & 33 \\
21 & 35 & 18 & 28 & 3 & 15
\end{bmatrix}
\end{align*}
(with $\mu=120$) is type $\typeA{\bP}$, where $\bP$ is given in (\ref{eq:convert:example}).   Using Theorem \ref{thm:convert:magic}, then the matrix
\begin{align*}
 \bA' & =\bQ \bA \bQ = 
\begin{bmatrix}
 32 & 16 & 4 & 27 & 39 & 2 \\
5 & 25 & 37 & 19 & 22 & 12 \\
23 & 7 & 31 & 14 & 11 & 34 \\
13 & 38 & 1 & 8 & 36 & 24 \\
26 & 6 & 29 & 17 & 9 & 33 \\
21 & 28 & 18 & 35 & 3 & 15
\end{bmatrix}
\end{align*}
where $\bQ$ is the permutation matrix in (\ref{ex:Q:transform}), 
is a type $\typeA{\bP'}$ magic square for $\bP'$ defined in (\ref{eq:convert:example}).  

%
%
%
%

%
%
%
%

\section{Singly-Even Ordered Magic Squares}

Many of the examples in this article use either order-4 magic squares, mainly due to the fact that there are only 7040 unique matrices, and such a small number allows analysis of all magic squares of this order.   Less work has been done for order-5 magic squares, but it is known that odd-ordered magic squares have significantly different properties from those of even-ordered ones, and as we will discuss in this section singly-even (those of order $4k+2$ for $k \in \mathbb{Z}^+$) magic squares  have different properties from those of doubly-even (order $4k$ for $k \in \mathbb{Z}^+$) magic squares.  

Also, it is noted that most of the order-6 magic squares that we have shown as examples here are not natural ones---instead we have selected squares that arise from removing 10, 20 and 30 from the set $\{1,2,3,\ldots, 38,39\}$.   Examples include those in (\ref{ex:typeA:6by6}) and (\ref{ex:typeB:6by6}).   The main reason for selecting a nonnatural magic square is the main result in this section, that no natural type A magic square exists.  

The main result below relies on a proof of Planck \cite{planck:1919}, who showed that a natural singly-even pandiagonal magic squares does not exist.  A consequence of this, as he pointed out, is that there are no regular magic squares.  As noted in the next theorem, this can be extended to any type A natural magic square.  

\begin{theorem}
 There are no type A natural magic squares of singly-even order. 
\end{theorem}


\begin{proof}
 Assume that $\bA$ is a type $\typeA{\bP}$ natural magic square of singly-even order.  From Theorem \ref{thm:convert:MCPMs} there is a matrix $\bQ$ such that $\bP = \bQ \bJ \bQ$. Then from Theorem \ref{thm:convert:magic}, $\bQ \bA \bQ$ is a type $\typeA{\bJ}$ (or regular) magic square, which is a contradiction, since Planck \cite{planck:1919} proved that no such natural regular magic square exists.   Therefore no type A natural magic square exists.   
 
\end{proof}

\section{Determinants, Eigenvalues and Eigenvectors of Even-Ordered Magic Squares}

We now explore determinants, eigenvalues and eigenvectors of type A and B magic squares.  In particular we will show that magic squares of these types are singular.  

\subsection{Type A Magic Squares}

In this section, we investigate type A magic squares, defined above.  The theorems presented here are similar to those in Mattingly \cite{mattingly:2000} but are generalized from a regular magic square to any type A magic.  First, we define, 
\begin{align}
\bZ = \bA - \frac{\mu}{n} \bE. \label{def:Z}
\end{align}
The matrix $\bZ$ is used in the proofs of theorems below, however is also useful for exploring the symmetry of a magic square.  For example, let
\begin{align}
\bA = 
\begin{bmatrix}
20&42&30&41&33&52&3&39\\
23&45&26&13&62&24&32&35\\
16&7&19&50&56&40&60&12\\
22&8&47&28&48&10&59&38\\
63&64&4&14&29&31&44&11\\
57&43&27&55&6&37&17&18\\
1&2&54&34&21&51&36&61\\
58&49&53&25&5&15&9&46\\
\end{bmatrix} \label{ex:typeA:8by8}
\end{align}
which is a type A magic square.  It satisfies $\bA + \bP\bA\bP = (2\mu/n) \bE$ where
\begin{align} 
\bP = 
\begin{bmatrix}
0&1&0&0&0&0&0&0\\
1&0&0&0&0&0&0&0\\
0&0&0&0&0&0&0&1\\
0&0&0&0&0&1&0&0\\
0&0&0&0&0&0&1&0\\
0&0&0&1&0&0&0&0\\
0&0&0&0&1&0&0&0\\
0&0&1&0&0&0&0&0\\
\end{bmatrix},
\label{eq:ex:perm:order8}
\end{align}
and $\mu=260$.  This permutation matrix can be constructed in a manner similar to that discussed in Lemma \ref{lemma:pn} and this is one of 7!! or 105 order-8 MCPMs.  

%
%


Using (\ref{def:Z}) and (\ref{ex:typeA:8by8}), this example results in 
\begin{align*}
	\bZ & = \left[\begin{array}{rrrrrrrr}
- 12.5& 9.5&- 2.5& 8.5&0.5& 19.5&- 29.5&6.5\\
 - 9.5&12.5&- 6.5&- 19.5& 29.5&- 8.5&- 0.5& 2.5\\ 
 - 16.5&- 25.5&- 13.5& 17.5& 23.5& 7.5& 27.5&-20.5\\
 - 10.5&-24.5& 14.5&- 4.5&15.5&- 22.5& 26.5&5.5\\
 30.5& 31.5&- 28.5&- 18.5&- 3.5&- 1.5& 11.5&-21.5\\
 24.5&10.5&- 5.5& 22.5&-26.5& 4.5&- 15.5&-14.5\\
 -31.5&-30.5& 21.5& 1.5&-11.5& 18.5& 3.5& 28.5\\ 
 25.5&16.5& 20.5&- 7.5&-27.5&- 17.5&- 23.5&13.5\\
\end{array}\right].
\end{align*}
The matrix $\bZ$ can be used to develop a Dudeney-like diagram for the matrix $\bP$ in (\ref{eq:ex:perm:order8}), which is fascinating in that there is no matrix $\bP$ in the definition of $\bZ$.  To create such a diagram, one draws links between numbers of the same magnitude and opposite signs---the diagram is found in Figure \ref{fig:dudeney:order8}.

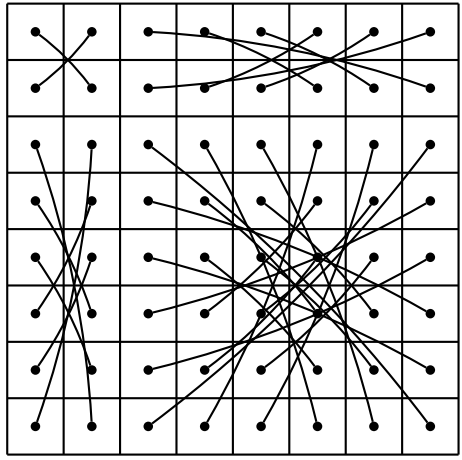
\begin{figure} 
\begin{center}
\psset{arcangle=8,unit=0.75} 
\begin{pspicture}(0,0)(8,8)
\multido{\i=0+1}{9}{\psline(0,\i)(8,\i)}
\multido{\i=0+1}{9}{\psline(\i,0)(\i,8)}

\dotnode(0.5,7.5){A11} \dotnode(1.5,7.5){A12}
\dotnode(2.5,7.5){A13} \dotnode(3.5,7.5){A14}
\dotnode(4.5,7.5){A15} \dotnode(5.5,7.5){A16}
\dotnode(6.5,7.5){A17} \dotnode(7.5,7.5){A18}

\dotnode(0.5,6.5){A21} \dotnode(1.5,6.5){A22}
\dotnode(2.5,6.5){A23} \dotnode(3.5,6.5){A24}
\dotnode(4.5,6.5){A25} \dotnode(5.5,6.5){A26}
\dotnode(6.5,6.5){A27} \dotnode(7.5,6.5){A28}

\dotnode(0.5,5.5){A31} \dotnode(1.5,5.5){A32}
\dotnode(2.5,5.5){A33} \dotnode(3.5,5.5){A34}
\dotnode(4.5,5.5){A35} \dotnode(5.5,5.5){A36}
\dotnode(6.5,5.5){A37} \dotnode(7.5,5.5){A38}

\dotnode(0.5,4.5){A41} \dotnode(1.5,4.5){A42}
\dotnode(2.5,4.5){A43} \dotnode(3.5,4.5){A44}
\dotnode(4.5,4.5){A45} \dotnode(5.5,4.5){A46}
\dotnode(6.5,4.5){A47} \dotnode(7.5,4.5){A48}

\dotnode(0.5,3.5){A51} \dotnode(1.5,3.5){A52}
\dotnode(2.5,3.5){A53} \dotnode(3.5,3.5){A54}
\dotnode(4.5,3.5){A55} \dotnode(5.5,3.5){A56}
\dotnode(6.5,3.5){A57} \dotnode(7.5,3.5){A58}

\dotnode(0.5,2.5){A61} \dotnode(1.5,2.5){A62}
\dotnode(2.5,2.5){A63} \dotnode(3.5,2.5){A64}
\dotnode(4.5,2.5){A65} \dotnode(5.5,2.5){A66}
\dotnode(6.5,2.5){A67} \dotnode(7.5,2.5){A68}

\dotnode(0.5,1.5){A71} \dotnode(1.5,1.5){A72}
\dotnode(2.5,1.5){A73} \dotnode(3.5,1.5){A74}
\dotnode(4.5,1.5){A75} \dotnode(5.5,1.5){A76}
\dotnode(6.5,1.5){A77} \dotnode(7.5,1.5){A78}

\dotnode(0.5,0.5){A81} \dotnode(1.5,0.5){A82}
\dotnode(2.5,0.5){A83} \dotnode(3.5,0.5){A84}
\dotnode(4.5,0.5){A85} \dotnode(5.5,0.5){A86}
\dotnode(6.5,0.5){A87} \dotnode(7.5,0.5){A88}

\ncarc{A11}{A22} \ncarc{A12}{A21}
\ncarc{A13}{A28} \ncarc{A14}{A26}
\ncarc{A15}{A27} \ncarc{A16}{A24}
\ncarc{A17}{A25} \ncarc{A18}{A23}

\ncarc{A31}{A82} \ncarc{A32}{A81}
\ncarc{A33}{A88} \ncarc{A34}{A86}
\ncarc{A35}{A87} \ncarc{A36}{A84}
\ncarc{A37}{A85} \ncarc{A38}{A83}

\ncarc{A41}{A62} \ncarc{A42}{A61}
\ncarc{A43}{A68} \ncarc{A44}{A66}
\ncarc{A45}{A67} \ncarc{A46}{A64}
\ncarc{A47}{A65} \ncarc{A48}{A63}

\ncarc{A51}{A72} \ncarc{A52}{A71}
\ncarc{A53}{A78} \ncarc{A54}{A76}
\ncarc{A55}{A77} \ncarc{A56}{A74}
\ncarc{A57}{A75} \ncarc{A58}{A73}

\end{pspicture}
\end{center}

\caption{A Dudeney diagram for a magic square of type $\typeA{\bP}$, where $\bP$ is found in (\ref{eq:ex:perm:order8}).  The magic square $\bA$ in (\ref{ex:typeA:8by8}) has a symmetry in which the entries connected by the arcs in the diagram add to $2\mu/n=65$ in this example.} 
\label{fig:dudeney:order8}
\end{figure}

This next theorem shows the eigenvalues of $\bZ$ are related to those of $\bA$.  

\begin{theorem} \label{lem:same:eigenvalues}
Let $\bA$ be a type A even-ordered magic square and let $\bZ$ be the matrix defined in (\ref{def:Z}).  Then $\bA$ and $\bZ$ have the same eigenvalues, except that $\mu$ is replaced by 0 in the spectrum of $\bZ$.   Also, both $\bZ$ and $\bA$ have the same eigenvectors.  
\end{theorem}

The proof of this theorem can be found in Mattingly \cite{mattingly:2000}.  His proof applied only to regular magic squares, however with minor changes, this extends to any type A magic square.  

\begin{theorem} \label{thm:eigen:plusminus}
Let $\bA$ be a type $\typeA{\bP}$ magic square.  If $\lambda \neq \mu$ is a eigenvalue of $\bA$ associated  with eigenvector $\bx$, then $-\lambda$ is also an eigenvalue with associated eigenvector $\bP \bx$.  
\end{theorem}

 \begin{proof} 

Let $\bZ$ be the matrix defined in (\ref{def:Z}).  First, we will show that $\bP \bZ \bP=-\bZ$. Left- and right-multiplication of  (\ref{def:Z}) by $\bP$ results in, 
\begin{align}
 \bP \bZ \bP &  = \bP \bA \bP - \frac{\mu}{n} \bP \bE \bP, \label{eq:PZP}
 \intertext{and using (\ref{eq:typeA}) as well as $\bP \bE \bP=\bE$.  Adding (\ref{def:Z}) and (\ref{eq:PZP}) and using $\bP\bE\bP=\bE$ results in } 
\bZ + \bP \bZ \bP & = \bA - \frac{\mu}{n} \bE  + \bP\bA\bP  -\frac{\mu}{n} \bE.  \label{eq:PZP2} 
\end{align}
Using (\ref{eq:typeA}), the right hand side of (\ref{eq:PZP2}) is 0, so $\bP\bZ\bP=-\bZ$.  
Since $\bP = \bP^{-1}$, $\bZ$ is similar to $-\bZ$, thus $\bZ$ and $-\bZ$ have the same eigenvalues.    A consequence of this is that if $\lambda \neq 0$ is an eigenvalue of $\bZ$, then $-\lambda$ is also an eigenvalue of $\bZ$.  
  
Next, assume that $\bx$ is an eigenvector of $\bZ$ with eigenvalue $\lambda$.  Multiply $\bP\bZ\bP=-\bZ$ on the right by $\bx$ to get 
 \begin{align*}
 \bP\bZ\bP\bx &= -\bZ\bx = -\lambda \bx\intertext{and multiplying this  through by $\bP^{-1}=\bP$, the result is}
 \bZ\bP\bx&=\bP^{-1}(-\lambda \bx) = -\lambda \bP \bx,
\end{align*}
so $-\lambda$ is an eigenvalue of $\bZ$ with related eigenvector $\bP\bx$.  Lastly since $\bZ$ and $\bA$ have the same eigenvalues (except for $\mu$) and eigenvectors as seen in Theorem \ref{lem:same:eigenvalues}, then the statement of the theorem is true.  
\end{proof}

In light of Theorems \ref{lem:same:eigenvalues} and \ref{thm:eigen:plusminus}, we examine $\bA$ in (\ref{ex:typeA:8by8}) in more detail.  The eigenvalues\footnote{The eigevalues in this example and later related eigenvectors are calculated numerically.  The results presented are rounded to either 3 or 4 significant digits to illustrate the theorems.} are
\begin{align*}
\{260,0,- 61.80,61.80,40.17\,i,- 40.17\,i,11.39\,i,- 11.39\,i\} 
\end{align*}
and interesting properties include:
\begin{itemize}
	\item The number 260 is the magic eigenvalue as is stated in Lemma \ref{lem:magic:eigenvalue}
	\item The matrix is singular since 0 is an eigenvalue.  We will prove this about all type A magic squares below.  
	\item Each of the other eigenvalues arise as pairs of equal magnitude and opposite sign.  This is a consequence of Theorem \ref{thm:eigen:plusminus}.  
\end{itemize}




If we extend the example shown above to include eigenvectors, consider the eigenvector\footnote{The eigenvectors shown here are found numerically as was the eigenvalues and are selected such that the 2-norm equals 1 with the same signs.}  associated with $\lambda_4 = 61.80$ which is 
\begin{align*}
\bx_4 = (0.230, 0.717,-0.355,-0.342,0.237,-0.020, -0.344,-0.123)^T
\end{align*}
and the eigenvector associated with $\lambda_3=-\lambda_4=-61.80$ is
\begin{align*}
 \bx_3 &=   	
(0.717,0.230,-0.123,-0.020,-0.344,-0.342,0.237,-0.355)^T 
\end{align*}
and it can be seen that $\bx_3=\bP \bx_4$ with $\bP$ defined in (\ref{eq:ex:perm:order8}), recalling that  left-multiplying by $\bP$ interchanges pairs of elements in a vector.  The remaining pairs of eigenvalues and eigenvectors are related in a similar manner.

\begin{theorem} \label{thm:Jordan:blocks}
	Suppose that a square matrix $\bB$ is similar to $-\bB$.  If a Jordan Block $\bJ_{{k}} \left( \lambda \right)$ appears $m$ times in the Jordan Canonical Form of $\bB$ then so does $\bJ_{{k}} \left(-\lambda \right)$.
\end{theorem}

Again, this theorem comes from Mattingly \cite{mattingly:2000} and is true because both $\bB$ and $-\bB$ have the same Jordan Form.  The following now extends Mattingly's proof that regular magic squares are singular to any type A magic square.

\begin{theorem}
If $\bA$ is a type A magic square of even-order, then $\det(\bA)=0$. 
\end{theorem}

\begin{proof}
 Let $\bZ$ be defined in (\ref{def:Z}).   For each nonzero eigenvalue $\lambda$ of $\bZ$, Theorem \ref{thm:Jordan:blocks}  ensures that the Jordan blocks of $\lambda$ are paired with those of $-\lambda$.
  
Then the sum of the algebraic multiplicities of the nonzero eigenvalues of $\bZ$ must even. Note that 0 is eigenvalue of $\bZ$, so itÕs multiplicity must also be even.  If $\bZ$ has at least two eigenvalues that are 0, this implies at least one eigenvalue of $\bA$ is 0. Thus, $\det(\bA)=0$. 
\end{proof}

\subsection{Type B Magic Squares}

\begin{theorem}
	If $\bA$ is an order-$2n$ magic square that satisfies either (\ref{def:typeBi}) or (\ref{def:typeBii}) then $\bA$ is singular.  Furthermore, the rank of $\bA$ is at most $n+1$.  
\end{theorem}

\begin{proof}
We first consider a matrix $\bA$ that satisfies (\ref{def:typeBi}) and let $P_i$ be the column where the one appears in the $i$th row of the matrix $\bP$ in (\ref{def:typeBi}). There are $n$ pairs $(i,P_i)$ with $i<P_i$.  Thus the $n$ row operations $R_i + R_{P_i} \rightarrow R_{P_i}$ results in $n$ rows contain the row vector $(2\mu/n) \be^T$.  

Since there are $n$ identical rows, the determinant is 0.  Also row operations can be used to get $n-1$ rows of zeros, thus the dimension of the null space is at least $n-1$ and thus the rank of the matrix can be no larger than $n+1$.   

If $\bA$ satisfies (\ref{def:typeBii}) then the same argument can be made using column operations to generate $n$ columns containing the vector $(\mu/n) \be$.  Similarly the rank of these matrices is also at most $n+1$.  

\end{proof}

We demonstrate this result with a natural 8 by 8 magic square.  The following matrix 

\begin{align}
\begin{bmatrix}
14&39&63&16&33&24&40&31\\
56&47&42&19&15&60&1&20\\
44&54&10&7&53&22&8&62\\
59&17&4&35&52&37&29&27\\
6&48&61&30&13&28&36&38\\
21&11&55&58&12&43&57&3\\
9&18&23&46&50&5&64&45\\
51&26&2&49&32&41&25&34\\ 
\end{bmatrix},
\label{eq:8by8:typeB}
\end{align}
is a type B magic---that satisfies (\ref{def:typeBi})---with $\bP=\bJ$.  The magic number in this case is $\mu=260$.  Applying the row operations  $R_1 + R_8 \rightarrow R_8, R_2 + R_7 \rightarrow R_7, R_3 + R_6 \rightarrow R_6$ and $R_4 + R_5 \rightarrow R_5$ leads to
\begin{align*}
\begin{bmatrix}
14&39&63&16&33&24&40&31\\
56&47&42&19&15&60&1&20\\
44&54&10&7&53&22&8&62\\
59&17&4&35&52&37&29&27\\
65 & 65 & 65 & 65 & 65 & 65 & 65 & 65\\
65 & 65 & 65 & 65 & 65 & 65 & 65 & 65\\
65 & 65 & 65 & 65 & 65 & 65 & 65 & 65\\
65 & 65 & 65 & 65 & 65 & 65 & 65 & 65\\
\end{bmatrix}.
\end{align*}
If we also perform the row operations $-R_5+R_6\rightarrow R_6, -R_5+R_7 \rightarrow R_7$ and $-R_5  + R_8 \rightarrow R_8$, the last 3 rows are zeroed out.  This shows that the rank of the matrix is at most 5.  

Also, the eigenspectrum of (\ref{eq:8by8:typeB}) is $\{260,0,0,0,-53.8553,49.6710, 2.0921\pm 43.6941 i\}$.  Recall that the number of zero eigenvalues of a matrix is equal to the dimension of the null space of the matrix and for an order-$n$ square matrix $\bA$, $\dim(\bA)+\rank(\bA)=n$, so the rank of this matrix is 5.

\section{Magic Square classification and eigenspectra}

Above we discussed that the rotations and reflections of a magic square as well as left- and right-multiplying by a bisymmetric and 90\textdegree-symmetric permutation matrix generates a family of magic squares, however only the bisymmetric ones apply to the next theorem.  

\begin{theorem}
Let $\{\bP_i\}$ be the set of bisymmetric permutation matrices that is the same size as a magic square $\bA$.  Then

\begin{itemize}
	\item $\bA, \bA^T, \bP_i\bA \bP_i, \bP_i \bA^T \bP_i$ each have the same eigenspectrum for each $i$.
	\item $\bJ \bP_i \bA \bP_i, \bP_i \bA \bP_i \bJ, \bP_i \bA^T \bP_i \bJ, \bJ \bP_i \bA^T \bP_i$ each has the same eigenspectrum for each $i$.\end{itemize}	
\end{theorem}

\begin{proof}
	Recall that  $\bA$ and $\bA^T$ have the identical eigenspectrum $\{ \lambda \}$. 
	
Since $\bP_i$ is a symmetric permutation matrix, it satifies $\bP_i=\bP_i^{-1}$ so $\bP_i\bA\bP_i$ is similar to $\bA$ and thus has the same eigenvalues.  $\bP_i \bA^T \bP_i$ is also similar to $\bA^T$, so also has the same eigenvalues.  

We now show the latter four matrices listed in the theorem have the same eigenspectrum.  Let the eigenspectrum of $\bJ\bP_i\bA\bP_i$ be $\{ \lambda\}$.  Since $\bJ$ satisfies $\bJ^{-1}=\bJ$, then left- and right-multiplying $\bJ\bP_i\bP\bP_i$ by $\bJ$ is $\bP_i \bP \bP_i \bJ$, therefore this has the eigenspectrum $\{\lambda\}$.  

Lastly, the transpose of $\bJ \bP_i \bA \bP_i$ is
\begin{align*}
( \bJ \bP_i \bA \bP_i)^T & = \bP_i^T \bA^T \bP_i^T \bJ^T = \bP_i \bA^T \bP_i \bJ
\end{align*}
and it has the same eigenspectrum of $\bJ \bP_i \bA \bP_i$.  A similar analysis shows $\bJ \bP_i \bA^T \bP_i$ has the same eigenspectrum.

\end{proof}

This theorem proves that half of the magic squares in a set of this form of magic squares have one set of eigenvalues and the other half has a (possibly) different set.

\section{Conclusions}

The seminal work of Dudeney on categorization of order-4 magic squares and the subsequent work of Trigg on the determinant of each order-4 magic square can be put into a greater framework.  We have extended the family of 8 magic squares formed by rotation and reflection to include transformations by classes of permutation matrices.  In addition, many of the types of order-4 magic squares that Dudeney described can be naturally extended to higher order magic squares. 

As has been noted in many works on magic squares, there are fundamental differences between even and odd-ordered magic squares and there are further differences between singly- and doubly-even magic squares.   We extended the proof that there are no natural singly-even regular magic squares to the fact that there are no natural singly-even type A magic squares.  

In addition, we generalized the work of Mattingly, who show that all even-ordered regular magic squares are singular.  We extended this to all even-ordered types A and B magic squares as well showing some properties of eigenvalues and eigenvectors of type A magic squares.

\subsubsection*{Acknowledgements}

The first author would first like to thank his former students Raymond Macon and Andy Gill, who started getting him thinking about magic squares.  He would also like to thank all of the students in his Mathematics Seminar class at Fitchburg State College during the Spring of 2010.  Some contributed directed to this paper (the last two authors) and everyone in the course contributed indirectly into our thinking.


\begin{thebibliography}{10}

\bibitem{dudeney:1917}
H.~E. Dudeney.
\newblock {\em Amusements in Mathematics}.
\newblock Dover, reprint edition (based on original 1917 manuscript published
  by thomas nelson, london, 1917) edition, 1970.

\bibitem{gardner:1976}
Martin Gardner.
\newblock Mathematical games.
\newblock {\em Scientific American}, 234:118--122, January 1976.

\bibitem{golub:1989}
Gene~H. Golub and Charles F.~Van Loan.
\newblock {\em Matrix Computations}.
\newblock Johns Hopkins, 2nd edition, 1989.

\bibitem{goodaire:2006}
Edgar~G. Goodaire and Michael~M. Parmenter.
\newblock {\em Discrete Mathematics with Graph Theory}, pages 322--324.
\newblock Prentice Hall, 3rd edition, 2006.

\bibitem{lang:1987}
Serge Lang.
\newblock {\em Linear Algebra}.
\newblock Springer-Verlag, 3rd edition, 1987.

\bibitem{loly:2009}
Peter Loly, Ian Cameron, Walter Trump, and Daniel Schindel.
\newblock Magic square spectra.
\newblock {\em Linear Algebra and its Applications}, 430(10):2659--2680, May
  2009.

\bibitem{mattingly:2000}
R.~Bruce Mattingly.
\newblock Even order regular magic squares are singular.
\newblock {\em American Mathematical Monthly}, 107(9), November 2000.

\bibitem{Pasles:2008}
Paul~C. Pasles.
\newblock {\em Benjamin Franklin's numbers: an unsung mathematical odyssey}.
\newblock Princeton University Press, Princeton, N.J., 2008.

\bibitem{pickover:2002}
C.~Pickover.
\newblock {\em The Zen of Magic Squares, Circles and Stars}.
\newblock Princeton University Press, Princeton, NJ, 2002.

\bibitem{pinn:1998}
Klaus Pinn and Christian Wieczerkowski.
\newblock Number of magic squares from parallel tempering monte carlo.
\newblock {\em Int. J. Mod. Phys. C}, 9:541--547, 1998.

\bibitem{planck:1919}
C.~Planck.
\newblock Pandiagonal magic squares of order 6 and 10 with minimal numbers,.
\newblock {\em The Monist}, 29:307--316, 1919.

\bibitem{trigg:1948}
C.~W. Trigg.
\newblock Determinants of fourth-order magic squares.
\newblock {\em Mathematical Monthly}, 1948.

\end{thebibliography}
\end{document}